\documentclass{article}

\usepackage[colorlinks,linkcolor=black,citecolor=black]{hyperref}
\usepackage{bm}
\usepackage{bbm}
\usepackage{mathrsfs}
\usepackage{amsmath}
\usepackage{amsthm}
\usepackage{amssymb}
\usepackage{mathdots}
\usepackage{multirow}
\usepackage{geometry}
\usepackage{mathtools}
\usepackage{fancyhdr}
\usepackage{authblk}
\usepackage[pagewise]{lineno}
\usepackage{indentfirst}
\usepackage[backend=biber,style=numeric,giveninits=true,sorting=nyt,maxbibnames=99]{biblatex}
\allowdisplaybreaks
\numberwithin{equation}{section}
\theoremstyle{plain}
\newtheorem{theorem}{Theorem}[section]
\newtheorem{proposition}[theorem]{Proposition}
\newtheorem{lemma}[theorem]{Lemma}

\newcommand{\re}{\operatorname{Re}}

\newcommand{\op}{\operatorname{Op}}
\newcommand{\sgn}{\operatorname{sgn}}
\newcommand{\ran}{\operatorname{Ran}}

\DeclareFieldFormat[article]{title}{#1}
\DeclareFieldFormat[article]{pages}{#1}
\DeclareFieldFormat[article]{number}{no\adddot\space#1}
\DeclareFieldFormat[article]{eid}{Art\adddot\space#1}

\renewbibmacro*{in:}{%
	\ifentrytype{article}
	{}
	{\printtext{\bibstring{in}\intitlepunct}}}
\renewbibmacro*{journal+issuetitle}{%
	\usebibmacro{journal}%
	\setunit*{\addspace}%
	\printfield{volume}%
	\setunit*{\addspace}%
	\usebibmacro{issue+date}%
	\setunit{\addcomma\space}%
	\printfield{number}%
	\setunit{\addcomma\space}%
	\printfield{eid}%
	\newunit}

\title{Linear Stability, Inviscid Damping, and Scattering for Monotone Shear Flows in 2D Ideal MHD}
\author{Zhile Li \thanks{School of Mathematical Sciences, University of Science and Technology of China, Hefei 230026, Anhui, China. Email: \texttt{lizhile@mail.ustc.edu.cn}}}
\date{}

\begin{document}
	\maketitle
	
	\begin{abstract}
		We study the linearized 2D ideal incompressible magnetohydrodynamic equations around a monotone shear flow \((U(y),0)\) with a constant horizontal magnetic field \((\beta,0)\) on \(\mathbb{T} \times \mathbb{R}\). For each fixed \(\beta \neq 0\), we assume that the shear rate \(U'\) is bounded above and below by positive constants and varies sufficiently slowly, although it need not be close to any constant. For every nontrivial perturbation with zero horizontal average, we establish sharp \(L^2\) time rates in the corresponding paired norms: the horizontal velocity and magnetic components remain of order one, the vertical components decay at rate \(\langle t \rangle^{-1}\), and the vorticity--current pair grows at rate \(\langle t \rangle\). We prove strong \(L^2\) scattering to two Alfv\'en profiles, identify the leading asymptotic behavior of the physical variables, and establish asymptotic completeness in the natural energy space.
	\end{abstract}
	
	\section{Introduction}
	
	The ideal incompressible magnetohydrodynamic (MHD) equations describe the motion of a perfectly conducting fluid coupled to a magnetic field. They are used, in particular, to study the dynamics of magnetized plasmas. In this paper, we consider these equations on \(\mathbb{T}_x \times \mathbb{R}_y\):
	\begin{equation}
		\label{eq:MHD}
		\begin{dcases}
			\partial_t \widetilde v + \widetilde v \cdot \nabla \widetilde v + \nabla \widetilde p = \widetilde b \cdot \nabla \widetilde b,\\
			\partial_t \widetilde b + \widetilde v \cdot \nabla \widetilde b = \widetilde b \cdot \nabla \widetilde v,\\
			\nabla \cdot \widetilde v = \nabla \cdot \widetilde b = 0.
		\end{dcases}
	\end{equation}
	Here \(\widetilde v\) and \(\widetilde b\) denote the velocity and magnetic field, respectively, and \(\widetilde p\) is the total pressure, including the magnetic pressure.
	
	A class of steady solutions of \eqref{eq:MHD} is given by
	\[v_E = (U(y),0), \qquad b_E =(\beta,0), \qquad p_E = \mathrm{constant},\]
	where \(U\) is a shear profile and \(\beta\) is a constant. Such equilibria represent flows along a uniform magnetic field whose speed varies across the field lines. They provide a simple setting in which to study shear mixing, magnetic coupling, and Alfv\'en-wave propagation.
	
	The linear stability of these equilibria is also of physical importance. Here, stability means that the velocity and magnetic perturbations governed by the linearized equations remain uniformly bounded in \(L^2\) as time evolves. To investigate linear stability around a monotone shear flow in \eqref{eq:MHD}, we write
	\[\widetilde v = v_E + v,\qquad \widetilde b = b_E + b,\qquad \widetilde p = p_E + p.\]
	Substituting these expressions into \eqref{eq:MHD} and neglecting the nonlinear terms in the perturbations, we obtain the linearized system:
	\begin{equation}
		\label{eq:linear}
		\begin{dcases}
			\partial_t v + U(y) \partial_x v + U'(y) v_2 e_1 + \nabla p = \beta \partial_x b,\\
			\partial_t b + U(y) \partial_x b - U'(y) b_2 e_1 = \beta\partial_x v,\\
			\nabla \cdot v = \nabla \cdot b = 0,
		\end{dcases}
	\end{equation}
	where \(e_1 = (1,0)\). Thus, the linear stability problem reduces to studying the solutions of \eqref{eq:linear}.
	
	Early normal-mode and spectral studies of ideal MHD shear flows revealed both stabilizing and destabilizing effects of parallel magnetic fields. In finite channels, Stern \cite{stern1963} showed that a suitably varying field can destabilize plane Couette flow, although the flow is stable without the field. Kent \cite{kent1968} showed that a constant field can stabilize some monotone velocity shears but destabilize others. Thus, monotonicity alone does not ensure linear stability in general channel geometries. For antisymmetric velocity shears in a rigid-wall channel, Chen--Morrison \cite{chen1991} derived a sufficient condition for instability and showed that a suitable symmetric magnetic field can destabilize plane Couette flow. Hughes--Tobias \cite{hughes2001} extended Howard's semicircle theorem to MHD shear flows with aligned, spatially varying magnetic fields and obtained sufficient conditions for spectral stability.
	
	For the ideal equations linearized around the Couette flow \((y,0)\) with a constant horizontal magnetic field \((\beta,0)\), Liang--Wu--Zhai \cite{liang2026} studied the nonzero horizontal modes under the condition \(|\beta|>1/2\). They proved uniform \(L^2\) upper bounds for the horizontal velocity and magnetic components, \(\langle t\rangle^{-1}\) inviscid damping of the vertical components, and upper and lower bounds of order \(\langle t\rangle\) for the vorticity--current pair. For the nonlinear ideal system around the same equilibrium, Knobel \cite{knobel2026} proved stability up to times of order \(\varepsilon^{-1}\) for suitable Gevrey perturbations of size \(\varepsilon\) and every \(\beta\neq0\). For suitable initial data, the vorticity--current pair grows at a linear rate over this interval.
	
	The two-dimensional Couette equilibrium with a constant horizontal magnetic field has also been studied in dissipative MHD. With magnetic resistivity but no viscosity, Knobel--Zillinger \cite{knobel2023} found echo-driven norm inflation for a simplified linearization around nearby traveling-wave solutions, while Zhao--Zi \cite{zhao2024} proved nonlinear asymptotic stability of the Couette equilibrium when the field is sufficiently strong. Sobolev stability thresholds with positive viscosity and resistivity were obtained in \cite{dolce2024,knobel2025,jin2025,wang2026,knobel2025c}, with related results for directional dissipation in \cite{knobel2025a,wu2026}. Dolce--Knobel--Zillinger \cite{dolce2026} proved large norm inflation of the current in the viscous, nonresistive regime. For Couette flow with zero background magnetic field, Li--Wu--Wu--Xu \cite{li2026a} proved nonlinear stability when both dissipative effects act only in the vertical direction.
	
	In three dimensions, Liss \cite{liss2020} established a Sobolev stability threshold for equal viscosity and resistivity in a sufficiently strong uniform magnetic field whose direction satisfies a Diophantine condition. Subsequent works treated unequal coefficients, rationally aligned fields, and a field perpendicular to the shear plane \cite{rao2025,wang2026a,knobel2025b}.
	
	For more general ideal MHD equilibria in two dimensions with velocity \((u(y),0)\) and magnetic field \((b(y),0)\), Liu--Masmoudi--Zhai--Zhao \cite{liu2022b} proved damping of the vertical velocity and magnetic components for nonzero horizontal modes on \(\mathbb T^2\) under suitable assumptions. For velocity shears initially close to Couette, Chen--Zi \cite{chen2025} proved nonlinear stability in a dissipative setting with equal viscosity and resistivity and a sufficiently strong constant magnetic field.
	
	In addition to linear stability, it is natural to ask how solutions of \eqref{eq:linear} behave as \(t \to \infty\). The background shear transports perturbations at different speeds across streamlines, while the magnetic field induces Alfv\'en propagation. To identify the two transport branches, we consider the Els\"asser combinations \(v\pm b\). Adding and subtracting the evolution equations in \eqref{eq:linear} gives
	\[(\partial_t + (U(y) \mp \beta) \partial_x) (v \pm b) + U'(y) (v_2 \mp b_2) e_1 + \nabla p = 0.\]
	We call the branches with transport velocities \(U(y) - \beta\) and \(U(y) + \beta\) the two Alfv\'en channels. They propagate in opposite directions relative to the background flow, while the remaining terms still couple them.
	
	In this setting, the scattering problem asks whether, after removing the transport phases \(e^{-ik (U(y) \mp \beta) t}\) on each nonzero horizontal Fourier mode \(k\), the amplitudes of the two Alfv\'en channels converge strongly in \(L^2\) to time-independent profiles. A related question is how these profiles describe the leading large-time behavior of the velocity and magnetic field, as well as the vorticity and current density. A further issue is whether scattering is asymptotically complete, in the sense that every admissible pair of limiting profiles arises from a unique finite-energy initial perturbation.
	
	The dynamics and scattering of Alfv\'en waves have also been studied near equilibria with zero velocity and a constant magnetic field. For localized perturbations near a sufficiently strong field, Bardos--Sulem--Sulem \cite{bardos} proved global existence and showed that nonlinear interactions become asymptotically negligible, leading to linear Alfv\'en-wave behavior. Global small-data results that include the ideal equations were subsequently obtained by Wei--Zhang \cite{weizhang2017} and Cai--Lei \cite{cai2018}. In three dimensions, He--Xu--Yu \cite{he2018} established global nonlinear stability and constructed scattering fields along the two families of Alfv\'en characteristics. For the nonlinear ideal incompressible MHD equations on \(\mathbb R^3\), Li--Yu \cite{li2021} proved that vanishing scattering fields on a suitable pair of characteristic infinities force the Alfv\'en-wave perturbation to vanish identically. Li \cite{li2023} proved local invertibility of the maps from initial data to suitable pairs of scattering fields in weighted Sobolev spaces. Li later proved unilateral rigidity: vanishing of the scattering field of either Alfv\'en-wave family forces that family to vanish identically \cite{li2024}. Duan--Li \cite{duan2026} obtained quantitative estimates for approximating scattering fields from finite-time data.
	
	Around nonuniform magnetic fields, the long-time dynamics can take other forms. Classical analyses of MHD-wave damping by phase mixing include the work of Tataronis--Grossmann \cite{tataronis1973}. For a zero-velocity equilibrium with a strictly monotone positive horizontal magnetic field in a finite channel, Ren--Zhao \cite{ren2017} proved strong \(L^2\) convergence of the horizontal Els\"asser combinations in their respective characteristic coordinates. For equilibria with both velocity and magnetic shear, Hirota--Tatsuno--Yoshida \cite{hirota2005} analyzed contrasting magnetic-island regimes for affine profiles. These regimes were subsequently studied rigorously in finite channels: Zhai--Zhang--Zhao \cite{zhai2021} proved that magnetic islands appear in the limiting state for suitable data when the profiles have a common zero and magnetic shear strictly dominates, whereas Ren--Wei--Zhang \cite{ren2021} established spectral stability and magnetic-island destruction for affine profiles when flow shear dominates. Under an additional compatibility condition at the magnetic neutral line, they also proved strong \(L^2\) scattering of the horizontal Els\"asser combinations. In the present problem, the monotone velocity shear determines the spatially varying transport phases of both Alfv\'en channels.
	
	We study a class of monotone shear flows for which the shear rate varies slowly in space. More precisely, fix \(\beta \neq 0\) and assume that
	\[0 < U'_{\inf} \le U'(y) \le U'_{\sup} < \infty, \qquad \| U'' \|_{W^{n,\infty}} \le \epsilon,\]
	where \(n\) is a sufficiently large fixed integer and \(\epsilon = \epsilon(U'_{\inf},U'_{\sup},\beta) > 0\) is a sufficiently small threshold depending only on these three parameters. Slow variation does not require \(U'\) to be close to a constant. For example, consider the family of shear rates
	\[U'(y) = \lambda \left( 1 + \frac{1}{2} \tanh \left( \frac{y}{\ell} \right) \right), \qquad \lambda>0, \qquad \ell \ge1.\]
	For each fixed \(n\), these profiles satisfy
	\[\frac{\lambda}{2} \le U'(y) \le \frac{3 \lambda}{2}, \qquad \| U'' \|_{W^{n,\infty}} \lesssim \frac{\lambda}{\ell},\]
	while
	\[\inf_{c \in \mathbb R} \| U' - c \|_{L^\infty} = \frac{\lambda}{2}.\]
	Thus the slow-variation condition holds for sufficiently large \(\ell\), even though \(U'\) remains at least \(\frac{\lambda}{2}\) away from every constant in \(L^\infty\).
	
	To state the stability theorem in full, we introduce the vorticity \(\omega\) and current density \(j\) of the perturbations:
	\[\omega = \nabla^\perp \cdot v, \qquad j = \nabla^\perp \cdot b,\]
	where \(\nabla^\perp = (-\partial_y, \partial_x)\). We also introduce the natural energy space
	\[\mathcal H = \left\{ (v,b) \in (L^2_{\neq} (\mathbb T \times \mathbb R))^4: \nabla \cdot v = \nabla \cdot b = 0 \right\},\]
	where
	\[L^2_{\neq} (\mathbb T \times \mathbb R) = \left\{ f \in L^2(\mathbb T \times \mathbb R): \int_{\mathbb T} f(x,\cdot) dx = 0 \right\}.\]
	The stability theorem is as follows.
	
	\begin{theorem}[Stability and inviscid damping]
		\label{thm:stability}
		Let \(\beta \neq 0\), and assume that
		\[0 < U'_{\inf} \le U'(y) \le U'_{\sup} < \infty, \qquad \| U'' \|_{W^{n,\infty}} \le \epsilon,\]
		where \(n\) is a sufficiently large fixed integer and \(\epsilon = \epsilon(U'_{\inf},U'_{\sup},\beta) > 0\) is a sufficiently small threshold depending only on these three parameters. Let \((v_{\mathrm{in}},b_{\mathrm{in}}) \in \mathcal H \cap (H^2(\mathbb T\times\mathbb R))^4\) be nonzero. Then the solution \((v,b)\) of \eqref{eq:linear} with the initial data \((v_{\mathrm{in}},b_{\mathrm{in}})\), together with its vorticity and current density \((\omega,j)\), satisfies
		\[
		\begin{alignedat}{2}
			c_{\mathrm{in}} &\le \| (v_1,b_1) \|_{L_x^2L_y^2} (t) &&\le C_{\mathrm{in}},\\
			c_{\mathrm{in}} \langle t \rangle^{-1} &\le \| (v_2,b_2) \|_{L_x^2L_y^2} (t) &&\le C_{\mathrm{in}} \langle t \rangle^{-1},\\
			c_{\mathrm{in}} \langle t \rangle &\le \| (\omega,j) \|_{L_x^2L_y^2} (t) &&\le C_{\mathrm{in}} \langle t \rangle,
		\end{alignedat}
		\]
		for all \(t \ge 0\), where \(c_{\mathrm{in}}, C_{\mathrm{in}}>0\) are constants depending on \(U'_{\inf},U'_{\sup},\beta\) and the initial perturbation.
	\end{theorem}
	
	We next state the scattering result and the leading asymptotics of the physical variables. In what follows, \(k \in \mathbb Z \setminus \{0\}\) denotes a horizontal Fourier mode, and \(f_k\) denotes the corresponding Fourier coefficient of \(f\). We write \(\Delta_k = \partial_y^2 - k^2\) for the Laplacian on this mode, so that \((-\Delta_k)^{-1/2}\) is the Fourier multiplier in \(y\) with symbol \((k^2 + \eta^2)^{-1/2}\). To remove the transport phases of the two Alfv\'en channels while retaining their \(L^2\) size, we introduce the phase-corrected amplitudes \(F_\pm\) through their horizontal Fourier coefficients:
	\begin{equation}
		\label{eq:phase-corrected-amplitudes}
		F_{\pm, k}(t,y) = e^{ik (U(y) \mp \beta) t} (-\Delta_k)^{-1/2} (\omega\pm j)_k(t,y).
	\end{equation}
	Indeed, the Biot--Savart relation shows that this normalization leaves the \(L^2\) norm unchanged:
	\[\| F_{\pm,k} \|_{L_y^2} = \| (v \pm b)_k \|_{L_y^2}.\]
	
	\begin{theorem}[Scattering and leading asymptotics]
		\label{thm:scattering}
		Under the assumptions on \(U\) and \(\beta\) in Theorem \ref{thm:stability}, let \((v_{\mathrm{in}},b_{\mathrm{in}}) \in \mathcal H \cap (H^2(\mathbb T \times \mathbb R))^4\). The solution \((v,b)\) of the linearized system \eqref{eq:linear} with initial data \((v_{\mathrm{in}}, b_{\mathrm{in}})\) uniquely determines scattering profiles \(F_{\pm,\infty} \in L^2_{\neq}(\mathbb T \times \mathbb R)\) satisfying
		\[\| F_\pm (t) - F_{\pm,\infty} \|_{L_x^2L_y^2} \le C_{\mathrm{in}} \langle t \rangle^{-1}\]
		for all \(t \ge 0\), where \(C_{\mathrm{in}}>0\) depends on \(U'_{\inf},U'_{\sup},\beta\) and the initial data. Moreover, for \(k \neq 0\), the following asymptotics hold as \(t \to \infty\):
		\begin{align*}
			(v_1 \pm b_1)_k (t,y) &= -i \sgn(k) e^{-ik (U(y) \mp \beta) t} F_{\pm,\infty,k}(y) + O_{L_y^2} (C_{\mathrm{in}} t^{-1}),\\
			(v_2 \pm b_2)_k (t,y) &= - \frac{i \sgn(k)}{U'(y) t} e^{-ik (U(y) \mp \beta) t} F_{\pm,\infty,k}(y) + o_{L_y^2} (t^{-1}),\\
			(\omega\pm j)_k (t,y) &= |k| U'(y) t e^{-ik (U(y) \mp \beta) t} F_{\pm,\infty,k}(y) + o_{L_y^2} (t).
		\end{align*}
	\end{theorem}
	
	We also establish asymptotic completeness in the natural energy space.
	
	\begin{theorem}[Asymptotic completeness]
		\label{thm:completeness}
		Under the assumptions on \(U\) and \(\beta\) in Theorem \ref{thm:stability}, for every \((v_{\mathrm{in}},b_{\mathrm{in}}) \in \mathcal H\), the limits \(F_{\pm,\infty} = \lim\limits_{t \to \infty} F_\pm (t)\) exist strongly in \(L^2_{\neq}(\mathbb T \times \mathbb R)\). The scattering map
		\[\mathcal S: \mathcal H \longrightarrow L^2_{\neq}(\mathbb T \times \mathbb R) \times L^2_{\neq}(\mathbb T \times\mathbb R), \quad (v_{\mathrm{in}},b_{\mathrm{in}}) \longmapsto (F_{+,\infty},F_{-,\infty})\]
		is a bounded linear bijection with bounded inverse. More precisely, there exist constants \(c,C>0\), depending only on \(U'_{\inf},U'_{\sup}\) and \(\beta\), such that
		\[c \| (v_{\mathrm{in}},b_{\mathrm{in}}) \|_{L_x^2L_y^2} \le \| (F_{+,\infty},F_{-,\infty}) \|_{L_x^2L_y^2} \le C \| (v_{\mathrm{in}},b_{\mathrm{in}}) \|_{L_x^2L_y^2}.\]
	\end{theorem}
	
	To prove these results, we first pass to coordinates moving with the shear and introduce a corrected velocity that absorbs the shear term in the magnetic-field equation. This correction is inspired by the adapted velocity used by Knobel for Couette flow \cite{knobel2026}. The time derivative of the natural physical \(L^2\) energy contains terms with no definite sign. To construct a non-increasing energy when \(U'\) is nonconstant, we draw on the variable-coefficient energy framework of our previous work \cite{li2026} and introduce a time-dependent pseudodifferential weight. This weight commutes with the sheared derivative, and its time variation produces favorable negative terms. The slow-variation assumption allows the resulting commutator errors involving \(U'\) to be absorbed. The resulting energy is equivalent to the physical \(L^2\) energy and is non-increasing. As in \cite{li2026}, we then construct terminal-time-dependent higher- and lower-order energies. Interpolation of these energy estimates yields upper bounds for growth and decay. Time reversibility and the Biot--Savart relations yield the matching lower bounds.
	
	For scattering, we derive an evolution system for the phase-corrected amplitudes \(F_\pm\). The direct estimates leave a cross-channel term with only \(t^{-1}\) decay, multiplied by the oscillatory factor \(e^{\mp2ik\beta t}\). A normal-form correction cancels this term. The lower-order energy estimate and the operator bounds give \(t^{-1}\) decay for the correction and \(t^{-2}\) decay for the time derivatives of the corrected amplitudes. This yields strong \(L^2\) scattering at rate \(t^{-1}\). The higher- and lower-order estimates then identify the leading asymptotics of the velocity, magnetic field, vorticity, and current density. Finally, uniform energy bounds extend the scattering map to \(\mathcal H\); the reverse energy bound ensures a closed range, and a terminal-data construction proves surjectivity.
	
	We fix some notation used throughout the paper. Set \(\langle \xi \rangle = \sqrt{1 + \xi^2}\). We define the Fourier transforms in the \(x\) and \(y\) directions by
	\[\mathcal F_x f(k) = \frac{1}{\sqrt{2\pi}} \int_{\mathbb T} e^{-ikx} f(x) dx, \qquad \mathcal F_y f(\eta) = \frac{1}{\sqrt{2\pi}} \int_{\mathbb R} e^{-iy\eta} f(y) dy.\]
	For \(s,r\in\mathbb R\), we use the Sobolev spaces \(H_x^sH_y^r\), equipped with the norm
	\[\| f \|_{H_x^sH_y^r}^2 = \sum_{k \in \mathbb Z} \int_{\mathbb R} \langle k \rangle^{2s} \langle \eta \rangle^{2r} |\mathcal F_x \mathcal F_y f(k,\eta)|^2 d\eta.\]
	For brevity, we write \(H_{x,y}^2\) for \(H^2(\mathbb T \times \mathbb R)\). We omit the underlying domain from other norms when it is clear from context. For a vector-valued function \(u = (u_1, \cdots, u_m)\), we use the same space notation and understand \(\| u \|_X^2 = \| u_1 \|_X^2 + \cdots + \| u_m \|_X^2\). Since the horizontal Fourier modes are decoupled in \eqref{eq:linear}, we fix a nonzero mode \(k \in \mathbb Z \setminus \{0\}\) throughout the mode-by-mode analysis. Unless otherwise specified, the norms \(\| \cdot \|_{L^2}\) and \(\| \cdot \|_{H^r}\), as well as the inner product \(\langle \cdot, \cdot \rangle\), are taken with respect to the \(y\)-variable.
	
	We introduce the following notations for inequalities between two nonnegative quantities \(A,B\). We write \(A \lesssim B\) if \(A \le CB\) for some constant \(C > 0\). Unless otherwise specified, the implicit constant may depend on \(U'_{\inf}\), \(U'_{\sup}\), and \(\beta\), but is independent of \(k, t\) and \(\epsilon\). The notation \(A \approx B\) means that \(A \lesssim B\) and \(B \lesssim A\) both hold.
	
	We next recall some notation for pseudodifferential operators; see, e.g., \cite{taylor1991, wong2014}. Let \(a(y,\eta)\) be a symbol, and denote by \(\op(a)\) its left quantization, defined by
	\[\op(a)f(y) = \frac{1}{\sqrt{2\pi}} \int_{\mathbb{R}} e^{iy\eta} a(y,\eta) \mathcal{F}_y f(\eta) d\eta.\]
	Conversely, given a pseudodifferential operator \(A\), we denote its symbol by \(\sigma(A)\). For two symbols \(a\) and \(b\), we define their symbolic composition by
	\[a \# b = \sigma(\op(a) \circ \op(b)),\]
	where \(\circ\) denotes the composition of operators. More explicitly,
	\begin{equation}
		\label{eq:composition_formula}
		(a \# b) (y,\eta) = \frac{1}{2\pi} \int_{\mathbb{R}} \int_{\mathbb{R}} e^{-iz\lambda} a(y,\eta+\lambda) b(y+z,\eta) d\lambda dz.
	\end{equation}
	All integrals of this type are understood as oscillatory integrals. In particular, if \(a = a(y)\) depends only on \(y\) or \(b = b(\eta)\) depends only on \(\eta\), then
	\[a \# b = ab.\]
	We denote the symbolic commutator by
	\[[a,b]=a \# b - b \# a.\]
	
	The rest of the paper is organized as follows. Section \ref{sec2} introduces the sheared coordinates and corrected velocity, derives the transformed equations, and establishes the basic weighted energy estimate. Section \ref{sec3} develops the higher- and lower-order energy estimates and uses them to prove stability, inviscid damping, and the sharp time rates. Section \ref{sec4} proves scattering and identifies the leading asymptotics. Section \ref{sec5} establishes asymptotic completeness.
	
	\section{The transformed system and the basic energy estimate}
	\label{sec2}
	
	In this section, we reformulate \eqref{eq:linear} in coordinates moving with the shear and introduce a corrected velocity. We then construct a weighted energy equivalent to the \(L^2\) energy of the perturbations and prove that it is non-increasing.
	
	\subsection{Sheared coordinates}
	
	Because the coefficients in \eqref{eq:linear} depend only on \(y\), the horizontal Fourier modes evolve independently, and the zero horizontal averages of \(v\) and \(b\) are preserved. Since \(v\) and \(b\) are divergence free, we introduce a stream function \(\psi\) and a magnetic potential \(\varphi\) such that
	\[v = \nabla^\perp \psi, \qquad b = \nabla^\perp \varphi.\]
	Then the vorticity and current density satisfy
	\[\omega = \Delta \psi, \qquad j = \Delta\varphi.\]
	
	We now take the Fourier transform of \eqref{eq:linear} in \(x\). The zero mode is absent, so for each \(k \neq 0\) the equations become
	\begin{equation}
		\label{eq:linear-k}
		\begin{dcases}
			\partial_t v_k + ikU v_k + U' v_{2,k} e_1 + \nabla_k p_k = ik \beta b_k,\\
			\partial_t b_k + ikU b_k - U' b_{2,k} e_1 = ik \beta v_k,\\
			\nabla_k \cdot v_k = \nabla_k \cdot b_k = 0,
		\end{dcases}
	\end{equation}
	where \(\nabla_k = (ik,\partial_y)\).
	
	To remove the transport term induced by the background shear, we define the unitary multiplication operator \(\mathcal U(t) = e^{ikU(y)t}\) and introduce
	\begin{equation}
		\label{eq:shear-transform}
		(V,B,P,\Omega,J,\Psi,\Phi) = \mathcal U(t) (v_k,b_k,p_k,\omega_k,j_k,\psi_k,\varphi_k).
	\end{equation}
	In the mode-by-mode analysis below, we suppress the mode index \(k\) on the transformed variables for simplicity. This is the Fourier representation of the change of coordinates \(x \mapsto x + U(y)t\). The corresponding sheared derivative, gradient, rotated gradient, and Laplacian are
	\[
	\begin{alignedat}{2}
		D_t &= \mathcal U(t) \partial_y \mathcal U(t)^{-1} = \partial_y - ikU'(y)t, \qquad& \nabla_t &= \mathcal U(t) \nabla_k \mathcal U(t)^{-1} = (ik,D_t),\\
		\nabla_t^\perp &= \mathcal U(t) \nabla_k^\perp \mathcal U(t)^{-1} = (-D_t,ik), \qquad& \Delta_t &= \mathcal U(t) \Delta_k \mathcal U(t)^{-1} = D_t^2-k^2,
	\end{alignedat}
	\]
	where \(\nabla_k^\perp = (-\partial_y,ik)\).
	
	Applying \(\mathcal U(t)\) to \eqref{eq:linear-k} gives
	\begin{equation}
		\label{eq:transformed-linear}
		\begin{dcases}
			\partial_t V + U' V_2 e_1 + \nabla_t P = ik \beta B,\\
			\partial_t B - U' B_2 e_1 = ik \beta V,\\
			\nabla_t \cdot V = \nabla_t \cdot B = 0.
		\end{dcases}
	\end{equation}
	The potential representation becomes
	\begin{equation}
		\label{eq:potential-representation}
		V = \nabla_t^\perp \Psi, \qquad B = \nabla_t^\perp \Phi, \qquad \Omega = \nabla_t^\perp \cdot V = \Delta_t\Psi, \qquad J = \nabla_t^\perp \cdot B = \Delta_t \Phi.
	\end{equation}
	
	\subsection{Corrected velocity}
	
	The second component of the magnetic equation in \eqref{eq:transformed-linear}, together with \(B_2 = ik\Phi\), gives
	\[ik \partial_t \Phi = \partial_t B_2 = ik \beta V_2.\]
	Since \(k \neq 0\), we have
	\[\partial_t \Phi = \beta V_2.\]
	This suggests modifying \(V\) to absorb the terms involving \(U'\) in both evolution equations. We therefore introduce the corrected velocity
	\[\widetilde V = V + \frac{U'}{\beta} \Phi e_1.\]
	It follows that \eqref{eq:transformed-linear} becomes
	\begin{equation}
		\label{eq:corrected-system}
		\begin{dcases}
			\partial_t \widetilde V + \nabla_t P = ik \beta B,\\
			\partial_t B = ik \beta \widetilde V,\\
			\nabla_t \cdot \widetilde V = \frac{U'}{\beta} B_2, \qquad \nabla_t \cdot B = 0.
		\end{dcases}
	\end{equation}
	
	To obtain an equation for the pressure, we differentiate the constraint \(\nabla_t \cdot V = 0\) in time and obtain
	\[\nabla_t \cdot \partial_t V = ikU'V_2,\]
	where we used \(\partial_t D_t = -ikU'\). Taking the divergence of the velocity equation in \eqref{eq:transformed-linear} then gives
	\begin{equation}
		\label{eq:pressure}
		\Delta_t P = -2ikU' V_2 = -2ikU' \widetilde V_2.
	\end{equation}
	
	\subsection{The basic energy}
	
	We introduce the time-dependent operator
	\begin{equation}
		\label{eq:basic-weight}
		S = \mathcal U(t) \op \left( \exp \left( \gamma \arctan \frac{\eta}{k} \right) \right) \mathcal U(t)^{-1},
	\end{equation}
	where
	\[\gamma = 1 + \frac{(U'_{\sup})^2}{|\beta| U'_{\inf}}.\]
	We define the basic energy by
	\[\mathcal E = \frac{1}{2} \| S \widetilde V \|_{L^2}^2 + \frac{1}{2} \| S B \|_{L^2}^2,\]
	and its associated dissipation by
	\[\mathcal D = \left\| k (-\Delta_t)^{-1/2} S \widetilde V \right\|_{L^2}^2 + \left\| k (-\Delta_t)^{-1/2} S B \right\|_{L^2}^2.\]
	The basic energy estimate is as follows.
	
	\begin{proposition}[Basic energy estimate]
		\label{prop:basic-energy}
		Under the assumptions on \(U\) and \(\beta\) in Theorem \ref{thm:stability}, let \((V,B)\) solve \eqref{eq:transformed-linear}. Then
		\begin{equation}
			\label{eq:basic-energy-equivalence}
			\mathcal E \approx \| V \|_{L^2}^2 + \| B \|_{L^2}^2
		\end{equation}
		uniformly in \(t \ge 0\) and \(k \neq 0\), and
		\begin{equation}
			\label{eq:basic-energy-decay}
			\frac{d}{dt} \mathcal E \le -\frac{U'_{\inf}}{2} \mathcal D \le 0.
		\end{equation}
		Consequently,
		\begin{equation}
			\label{eq:mode-uniform-bound}
			\| (V,B) \|_{L^2} (t) \lesssim \| (V,B) \|_{L^2} (0).
		\end{equation}
	\end{proposition}
	
	To motivate the choice of \(S\), consider the unweighted \(L^2\) energy of \((\widetilde V,B)\). The corrected evolution equations in \eqref{eq:corrected-system} have no explicit shear terms, and the terms proportional to \(ik\beta\) cancel after taking real parts. A calculation similar to that in the proof of Proposition \ref{prop:basic-energy} shows that
	\[\frac{d}{dt} \frac{1}{2} \left( \| \widetilde V \|_{L^2}^2 + \| B \|_{L^2}^2 \right) = \frac{2}{\beta} \re \langle U' B_2, ik (-\Delta_t)^{-1} (U' \widetilde V_2) \rangle.\]
	The right-hand side has no definite sign. To control it, we seek a time-dependent weight \(S\) whose derivative produces a negative term with the same \((-\Delta_t)^{-1}\) structure. At the level of principal symbols, we require
	\[(\partial_t S) S^{-1} \approx - \gamma U' k^2 (-\Delta_t)^{-1}.\]
	Formally, this suggests taking
	\[S = \exp \left( \gamma \arctan \left( \frac{D_t}{ik} \right) \right),\]
	which agrees with the definition in \eqref{eq:basic-weight} and commutes with \(D_t\). The slow-variation assumption makes the resulting remainders small enough to be absorbed in the energy estimate.
	
	\subsection{Technical lemmas}
	
	To prove Proposition \ref{prop:basic-energy}, we first recall the following Calder\'on--Vaillancourt estimate \cite{hwang1987}.
	
	\begin{theorem}[Calder\'on--Vaillancourt]
		\label{thm:CV}
		Let \(a : \mathbb R \times \mathbb R\to\mathbb C\) be a continuous symbol satisfying
		\[\| \partial_y^j \partial_\eta^l a \|_{L_y^\infty L_\eta^\infty} < \infty, \qquad j,l \in \{0,1\}.\]
		Then
		\[\| \op(a) \|_{L^2 \to L^2} \lesssim \max_{j, l \in \{0,1\}} \| \partial_y^j \partial_\eta^l a \|_{L_y^\infty L_\eta^\infty}.\]
	\end{theorem}
	
	We next establish two technical lemmas. The first identifies the principal part of \((\partial_t S)S^{-1}\).
	
	\begin{lemma}[Time derivative of the basic weight]
		\label{lem:weight-derivative}
		Under the hypotheses on \(U\) and \(\beta\) of Theorem \ref{thm:stability}, we have the following decomposition:
		\[\sigma \left( \mathcal U^{-1} (\partial_t S) S^{-1} \mathcal U \right) (y,\eta) = - \gamma U'(y) \left \langle \frac{\eta}{k} \right \rangle^{-2} + r_S(y,\eta),\]
		where the remainder \(r_S\) satisfies
		\[\left\| \op \left( \left\langle \frac{\eta}{k} \right\rangle \# r_S \# \left\langle \frac{\eta}{k} \right\rangle \right) \right\|_{L^2 \to L^2} \lesssim \epsilon.\]
	\end{lemma}
	
	\begin{proof}
		Since \(\partial_t \mathcal U = ikU \mathcal U\), differentiating the definition of \(S\) gives
		\[\partial_t S = ikU S - ik S U = ik [U,S].\]
		Conjugating by \(\mathcal U^{-1}\) and using \(U \mathcal U = \mathcal U U\), we have
		\[\sigma \left( \mathcal U^{-1} (\partial_t S) S^{-1} \mathcal U \right) (y,\eta) = ik \left[U(y), \exp \left( \gamma \arctan \frac{\eta}{k} \right) \right] \exp \left( - \gamma \arctan \frac{\eta}{k} \right).\]
		
		Applying the symbolic composition formula \eqref{eq:composition_formula}, we obtain
		\[U(y) \# \exp \left( \gamma \arctan \frac{\eta}{k} \right) = U(y) \exp \left( \gamma \arctan \frac{\eta}{k} \right),\]
		and
		\[\exp \left( \gamma \arctan \frac{\eta}{k} \right) \# U(y) = \frac{1}{2\pi} \int_{\mathbb R} \int_{\mathbb R} e^{-iz \lambda} \exp \left( \gamma \arctan \frac{\eta + \lambda}{k} \right) U(y+z) d \lambda dz.\]
		We expand the exponential factor to first order in \(\lambda\), retaining the integral remainder:
		\[\exp \left( \gamma \arctan \frac{\eta + \lambda}{k} \right) = \exp \left( \gamma \arctan \frac{\eta}{k} \right) + \lambda \partial_\eta \exp \left( \gamma \arctan \frac{\eta}{k} \right) + \lambda^2 \int_{0}^{1} (1-\theta) \partial_\eta^2 \exp \left( \gamma \arctan \frac{\eta + \theta \lambda}{k} \right) d \theta.\]
		Substituting this expansion into the preceding formula and integrating by parts in \(z\), we obtain
		\begin{align*}
			\exp \left( \gamma \arctan \frac{\eta}{k} \right) \# U(y)
			=& U(y) \exp \left( \gamma \arctan \frac{\eta}{k} \right) -i U'(y) \partial_\eta \exp \left( \gamma \arctan \frac{\eta}{k} \right) \\
			&- \frac{1}{2\pi} \int_{\mathbb R} \int_{\mathbb R} \int_{0}^{1} e^{-iz\lambda} U''(y+z) (1-\theta) \partial_\eta^2 \exp \left( \gamma \arctan \frac{\eta + \theta \lambda}{k} \right) d \theta d \lambda dz.
		\end{align*}
		
		Thus we have the following decomposition:
		\[\sigma \left( \mathcal U^{-1} (\partial_t S) S^{-1} \mathcal U \right) (y,\eta) = - \gamma U'(y) \left \langle \frac{\eta}{k} \right \rangle^{-2} + r_S(y,\eta),\]
		where
		\begin{equation}
			\label{eq:r-remainder}
			r_S(y,\eta) = \frac{ik}{2\pi} \int_{\mathbb R} \int_{\mathbb R} \int_{0}^{1} e^{-iz\lambda} U''(y+z) (1-\theta) \left( \partial_\eta^2 \exp \left( \gamma \arctan \frac{\eta + \theta \lambda}{k}\right) \right) \exp \left( - \gamma \arctan \frac{\eta}{k} \right) d\theta d\lambda dz.
		\end{equation}
		
		To establish the operator bound in the lemma, we first estimate \(r_S\). Integrating by parts in the oscillatory variables, we rewrite \eqref{eq:r-remainder} as
		\begin{align*}
			r_S(y,\eta) =& \frac{ik}{2\pi} \int_{\mathbb R} \int_{\mathbb R} \int_{0}^{1} e^{-iz\lambda} \langle z \rangle^{-2} (1 - \partial_\lambda^2) \bigg[ \langle \lambda \rangle^{-4} (1 - \partial_z^2)^2 U''(y+z) \\
			&\times (1-\theta) \left( \partial_\eta^2 \exp \left( \gamma \arctan \frac{\eta + \theta \lambda}{k}\right) \right) \exp \left( - \gamma \arctan \frac{\eta}{k} \right) \bigg] d\theta d\lambda dz.
		\end{align*}
		Direct differentiation and the boundedness of \(\arctan\) imply, for \(m \in \{0, 1, 2\}\),
		\[\left| \partial_\lambda^m \partial_\eta^2 \exp \left( \gamma \arctan \frac{\eta + \theta \lambda}{k} \right) \right| \lesssim |k|^{-2} \left\langle \frac{\eta + \theta \lambda}{k} \right \rangle^{-2}, \qquad \left| \exp \left( - \gamma \arctan \frac{\eta}{k} \right) \right| \lesssim 1.\]
		By Peetre's inequality, we obtain
		\[\left\langle \frac{\eta + \theta \lambda}{k} \right\rangle^{-2} \lesssim \left\langle \frac{\eta}{k} \right\rangle^{-2} \left\langle \frac{\theta \lambda}{k} \right\rangle^2 \lesssim \left\langle \frac{\eta}{k} \right\rangle^{-2} \langle \lambda \rangle^2.\]
		Thus
		\[|r_S(y,\eta)| \lesssim \epsilon |k|^{-1} \int_{\mathbb R} \int_{\mathbb R} \langle z \rangle^{-2} \langle \lambda \rangle^{-4} \left\langle \frac{\eta}{k} \right\rangle^{-2} \langle \lambda \rangle^2 d \lambda dz \lesssim \epsilon \left\langle \frac{\eta}{k} \right\rangle^{-2}.\]
		The same argument applies after differentiating \eqref{eq:r-remainder}: the \(y\)-derivatives fall on \(U''\), while the \(\eta\)-derivatives of the exponential factors satisfy the same type of bounds. Hence, for \(j, l \in \{0, \cdots, 5\}\),
		\begin{equation}
			\label{eq:r-symbol-bound}
			\left| \partial_y^j \partial_\eta^l r_S(y,\eta) \right| \lesssim \epsilon \left \langle \frac{\eta}{k} \right\rangle^{-2}.
		\end{equation}
		
		Applying \eqref{eq:composition_formula} once more, we have
		\[\left\langle \frac{\eta}{k} \right\rangle \# r_S(y,\eta) \# \left\langle \frac{\eta}{k} \right\rangle = \frac{1}{2\pi} \int_{\mathbb R} \int_{\mathbb R} e^{-iz\lambda} \left\langle \frac{\eta + \lambda}{k} \right\rangle r_S(y+z,\eta) \left\langle \frac{\eta}{k} \right\rangle d\lambda dz.\]
		Peetre's inequality, \eqref{eq:r-symbol-bound}, and integration by parts in the oscillatory variables give
		\[\left| \partial_y^j \partial_\eta^l \left( \left\langle \frac{\eta}{k} \right\rangle \# r_S(y,\eta) \# \left\langle \frac{\eta}{k} \right\rangle \right) \right| \lesssim \epsilon, \qquad j,l \in \{0,1\}.\]
		The Calder\'on--Vaillancourt Theorem \ref{thm:CV} yields the claimed operator bound.
	\end{proof}
	
	The next lemma controls the conjugation of \(U'\) by a Fourier multiplier.
	
	\begin{lemma}[Conjugation of the shear coefficient]
		\label{lem:shear-conjugation}
		Under the assumptions on \(U\) in Theorem \ref{thm:stability}, let \(a = a(\eta) > 0\) be a Fourier symbol and set \(\mathcal A = \mathcal U\op(a)\mathcal U^{-1}\). Assume that
		\[|\partial_\eta^l a(\eta)| \lesssim a(\eta), \qquad l \in \{0, \cdots, 4\}.\]
		Suppose also that there is a fixed \(\alpha \ge 0\), with \(\alpha + 4 \le n\), such that
		\[\frac{a(\eta + \lambda)}{a(\eta)} \lesssim \langle \lambda \rangle^\alpha.\]
		Then
		\[\left\| \mathcal A U' \mathcal A^{-1} - U' \right\|_{L^2 \to L^2} \lesssim \epsilon.\]
	\end{lemma}
	
	\begin{proof}
		Since \(\mathcal U\) is unitary and commutes with \(U'\), it suffices to prove
		\[\left\| \op (a \# U' \# a^{-1} - U') \right\|_{L^2 \to L^2} \lesssim \epsilon.\]
		By the composition formula \eqref{eq:composition_formula}, the symbol is
		\[(a \# U' \# a^{-1} - U')(y,\eta) = \frac{1}{2\pi} \int_{\mathbb R} \int_{\mathbb R} e^{-iz\lambda} U'(y+z) \frac{a(\eta + \lambda) - a(\eta)}{a(\eta)} d\lambda dz.\]
		Since
		\[a(\eta + \lambda) - a(\eta) = \lambda \int_{0}^{1} a'(\eta + \theta \lambda) d \theta,\]
		integration by parts in \(z\) gives
		\begin{equation}
			\label{eq:shear-conjugation-symbol}
			(a \# U' \# a^{-1} - U')(y,\eta) = - \frac{i}{2\pi} \int_{\mathbb R} \int_{\mathbb R} \int_{0}^{1} e^{-iz\lambda} U''(y+z) \frac{a'(\eta + \theta \lambda)}{a(\eta)} d\theta d\lambda dz.
		\end{equation}
		For \(m \in \{0, 1, 2\}\) and \(l \in \{0, 1\}\), the assumptions on \(a\) imply
		\[\left| \partial_\eta^l \partial_\lambda^m \left( \frac{a'(\eta + \theta \lambda)} {a(\eta)} \right) \right| \lesssim \frac{a(\eta + \theta \lambda)}{a(\eta)} \lesssim \langle \lambda \rangle^\alpha.\]
		Choose an integer \(p\) such that \(\alpha + 1 < 2p \le n - 1\). Applying integration by parts to \eqref{eq:shear-conjugation-symbol}, using \(1 - \partial_\lambda^2\) to gain \(\langle z \rangle^{-2}\) and \((1 - \partial_z^2)^p\) to gain \(\langle \lambda \rangle^{-2p}\), we obtain
		\[|(a \# U' \# a^{-1} - U')(y,\eta)| \lesssim \epsilon \int_{\mathbb R} \int_{\mathbb R} \langle z\rangle^{-2} \langle \lambda \rangle^{-2p} \langle \lambda \rangle^\alpha d \lambda dz \lesssim \epsilon.\]
		The same argument applies after taking up to one derivative in each of \(y\) and \(\eta\). More precisely,
		\[|\partial_y^j \partial_\eta^l (a \# U' \# a^{-1} - U')(y,\eta)| \lesssim \epsilon, \qquad j, l \in \{0,1\}.\]
		Here at most \(2p + 1 \le n\) derivatives fall on \(U''\), while at most \(4\) derivatives fall on \(a\). The Calder\'on--Vaillancourt Theorem \ref{thm:CV} completes the proof.
	\end{proof}
	
	\subsection{Proof of Proposition \ref{prop:basic-energy}}
	
	With the preceding lemmas established, we turn to the proof of the basic energy estimate.
	
	\begin{proof}[Proof of Proposition \ref{prop:basic-energy}]
		We first prove \eqref{eq:basic-energy-equivalence}. By its definition in \eqref{eq:basic-weight}, \(S\) is the conjugate of the Fourier multiplier \(\exp( \gamma \arctan \frac{\eta}{k})\) by the unitary operator \(\mathcal U\). Since this multiplier is bounded above and below by positive constants, Plancherel's theorem gives
		\[\| Sf \|_{L^2} \approx \| f \|_{L^2}\]
		for all \(f \in L^2(\mathbb R)\). Moreover, \(B_2 = ik \Phi\), so the definition of \(\widetilde V\) implies
		\[\| \widetilde V - V \|_{L^2} = \left\| \frac{U'}{\beta} \Phi \right\|_{L^2} = \left\| \frac{U'}{ik \beta} B_2 \right\|_{L^2} \le \frac{U'_{\sup}} {|\beta|} \|B\|_{L^2}.\]
		Applying the triangle inequality in both directions, we obtain
		\[\left( 1 + \frac{U'_{\sup}}{|\beta|} \right)^{-2} \left( \| V \|_{L^2}^2 + \| B \|_{L^2}^2 \right) \le \| \widetilde V \|_{L^2}^2 + \| B \|_{L^2}^2 \le \left( 1 + \frac{U'_{\sup}}{|\beta|} \right)^2 \left( \| V \|_{L^2}^2 + \| B \|_{L^2}^2 \right).\]
		Together with the definition of \(\mathcal E\), \eqref{eq:basic-energy-equivalence} follows.
		
		To prove \eqref{eq:basic-energy-decay}, we compute the time derivative of \(\mathcal E\). By \eqref{eq:corrected-system},
		\[\frac{d}{dt}\mathcal E = \re \langle S \widetilde V, (\partial_t S) \widetilde V \rangle + \re \langle SB, (\partial_t S) B \rangle + \re \langle S \widetilde V, ik\beta S B - S \nabla_t P \rangle + \re \langle SB, ik \beta S \widetilde V \rangle.\]
		The terms involving \(ik\beta\) cancel because
		\[\re \langle S \widetilde V, ik \beta S B \rangle = - \re \langle SB, ik \beta S \widetilde V \rangle.\]
		Applying \(S D_t = D_t S\) and integrating by parts in the pressure term, we have
		\[- \re \langle S \widetilde V, S \nabla_t P \rangle = - \re \langle S \widetilde V, \nabla_t (SP) \rangle = \re \langle \nabla_t \cdot (S \widetilde V), SP \rangle = \frac{1}{\beta} \re \langle S(U'B_2), SP \rangle,\]
		where we used \(\nabla_t \cdot \widetilde V = \frac{U'}{\beta} B_2\). Consequently, combining with \eqref{eq:pressure}, we obtain
		\begin{equation}
			\label{eq:basic-energy-identity}
			\frac{d}{dt}\mathcal E = \re \langle S \widetilde V, (\partial_t S) \widetilde V \rangle + \re \langle S B, (\partial_t S) B \rangle + \frac{2}{\beta} \re \langle S (U' B_2), ik (-\Delta_t)^{-1} S (U' \widetilde V_2) \rangle.
		\end{equation}
		
		We now estimate the terms involving \(\partial_t S\). We give the details for the term containing \(\widetilde V\) and the term containing \(B\) is treated in the same way. The decomposition in Lemma \ref{lem:weight-derivative} gives
		\[(\partial_t S) S^{-1} = - \gamma U' k^2 (-\Delta_t)^{-1} + \mathcal R_S,\]
		where \(\mathcal R_S = \mathcal U\op(r_S)\mathcal U^{-1}\). Therefore,
		\[\re \langle S \widetilde V, (\partial_t S) \widetilde V \rangle = - \gamma \re \langle S \widetilde V, U' k^2 (-\Delta_t)^{-1} S \widetilde V \rangle + \re \langle S \widetilde V, \mathcal R_S S \widetilde V \rangle.\]
		
		To estimate the principal term, we apply Lemma \ref{lem:shear-conjugation} to \((-\Delta_t)^{1/2}\), whose symbol before conjugation is \((k^2 + \eta^2)^{1/2}\). This gives
		\[\left\| (-\Delta_t)^{1/2} U' (-\Delta_t)^{-1/2} - U' \right\|_{L^2 \to L^2} \lesssim \epsilon.\]
		Since \((-\Delta_t)^{-1/2}\) is self-adjoint, the principal term can be written as
		\begin{align*}
			- \gamma \re \langle S \widetilde V, U' k^2 (-\Delta_t)^{-1} S \widetilde V \rangle =& - \gamma \re \left\langle k (-\Delta_t)^{-1/2} S \widetilde V, (-\Delta_t)^{1/2} U' (-\Delta_t)^{-1/2} k (-\Delta_t)^{-1/2} S \widetilde V \right\rangle \\
			=& - \gamma \left\langle k (-\Delta_t)^{-1/2} S \widetilde V, U' k (-\Delta_t)^{-1/2} S \widetilde V \right\rangle \\
			&- \gamma \re \left\langle k (-\Delta_t)^{-1/2} S \widetilde V, \left( (-\Delta_t)^{1/2} U' (-\Delta_t)^{-1/2} - U' \right) k (-\Delta_t)^{-1/2} S \widetilde V \right\rangle \\
			\le& - \gamma (U'_{\inf} - C \epsilon) \left\| k  (-\Delta_t)^{-1/2} S \widetilde V \right\|_{L^2}^{2}.
		\end{align*}
		The last inequality follows from \(U' \ge U'_{\inf}\) and the preceding operator bound. Here and below, we use \(C > 0\) for constants that may change from line to line but depend only on \(U'_{\inf},U'_{\sup},\beta\), uniformly in \(k,t,\epsilon\).
		
		The remainder bound in Lemma \ref{lem:weight-derivative} yields
		\[\left\| k^{-1} (-\Delta_t)^{1/2} \mathcal R_S k^{-1} (-\Delta_t)^{1/2} \right\|_{L^2 \to L^2} = \left\| \op \left( \left\langle \frac{\eta}{k} \right\rangle \# r_S \# \left\langle \frac{\eta}{k} \right\rangle \right) \right\|_{L^2 \to L^2} \lesssim \epsilon,\]
		which implies
		\[\left| \re \langle S \widetilde V, \mathcal R_S S \widetilde V \rangle\right| = \left| \re \left\langle k (-\Delta_t)^{-1/2} S \widetilde V, k^{-1} (-\Delta_t)^{1/2} \mathcal R_S k^{-1} (-\Delta_t)^{1/2} k (-\Delta_t)^{-1/2} S \widetilde V \right\rangle \right| \lesssim \epsilon \left\| k (-\Delta_t)^{-1/2} S \widetilde V \right\|_{L^2}^{2}.\]
		Thus we conclude that
		\[\re \langle S \widetilde V, (\partial_t S) \widetilde V \rangle \le (-\gamma U'_{\inf} + C \epsilon) \left\| k (-\Delta_t)^{-1/2} S \widetilde V \right\|_{L^2}^{2}.\]
		The same argument applies to \(B\). Adding the two estimates and recalling the definition of \(\mathcal D\), we obtain
		\[\re \langle S \widetilde V, (\partial_t S) \widetilde V \rangle + \re \langle S B, (\partial_t S) B \rangle \le (-\gamma U'_{\inf} + C \epsilon) \mathcal D.\]
		
		It remains to estimate the pressure term in \eqref{eq:basic-energy-identity}. Applying Lemma \ref{lem:shear-conjugation} to \((-\Delta_t)^{-1/2} S\), we have
		\[\left\| (-\Delta_t)^{-1/2} S U' S^{-1} (-\Delta_t)^{1/2} - U' \right\|_{L^2 \to L^2} \lesssim \epsilon.\]
		Consequently,
		\begin{align*}
			\left\| k (-\Delta_t)^{-1/2} S (U' B_2) \right\|_{L^2} &\le \left\| (-\Delta_t)^{-1/2} S U' S^{-1} (-\Delta_t)^{1/2} \right\|_{L^2 \to L^2} \left\| k (-\Delta_t)^{-1/2} S B_2 \right\|_{L^2} \\
			&\le (U'_{\sup} + C \epsilon) \left\| k (-\Delta_t)^{-1/2} S B_2 \right\|_{L^2},
		\end{align*}
		where we used \(\| U' \|_{L^2 \to L^2} = U'_{\sup}\). The same estimate holds with \(B_2\) replaced by \(\widetilde V_2\). Hence, by the Cauchy--Schwarz inequality and Young's inequality, together with \(|k| \ge 1\), it follows that
		\begin{align*}
			\left| \frac{2}{\beta} \re \left\langle S(U'B_2), ik (-\Delta_t)^{-1} S(U' \widetilde V_2) \right\rangle \right| &\le
			\frac{2}{|\beta k|} \left\| k (-\Delta_t)^{-1/2} S (U'B_2) \right\|_{L^2} \left\| k (-\Delta_t)^{-1/2} S (U'\widetilde V_2) \right\|_{L^2} \\
			&\le \frac{2 (U'_{\sup} + C \epsilon)^2}{|\beta|} \left\| k (-\Delta_t)^{-1/2} S B_2 \right\|_{L^2} \left\| k (-\Delta_t)^{-1/2} S \widetilde V_2 \right\|_{L^2} \\
			&\le \frac{(U'_{\sup} + C \epsilon)^2}{|\beta|} \mathcal D \\
			&\le \left( \frac{(U'_{\sup})^2}{|\beta|} + C \epsilon \right) \mathcal D.
		\end{align*}
		In the last step, we used the convention that \(C\) may change from line to line.
		
		Combining \eqref{eq:basic-energy-identity} with the preceding estimates and using the definition of \(\gamma\), we obtain
		\[\frac{d}{dt} \mathcal E \le \left( -\gamma U'_{\inf} + \frac{(U'_{\sup})^2}{|\beta|} + C \epsilon \right) \mathcal D = \left(- U'_{\inf} + C \epsilon \right) \mathcal D \le - \frac{U'_{\inf}}{2} \mathcal D \le 0.\]
		Here the assumed smallness of \(\epsilon\) ensures \(C \epsilon \le \frac{U'_{\inf}}{2}\). This proves \eqref{eq:basic-energy-decay} and completes the proof.
	\end{proof}
	
	\section{Stability and inviscid damping}
	\label{sec3}
	
	The basic energy estimate gives uniform \(L^2\) control of \((V,B)\). We now introduce terminal-time-dependent weights to obtain the higher- and lower-order estimates needed for Theorem \ref{thm:stability}. In deriving these estimates, we fix a terminal time \(s \ge 0\) and set \(0 \le t \le s\). We then combine the energy estimates to prove stability and inviscid damping.
	
	\subsection{The higher-order energy estimate}
	
	For the higher-order estimate, we choose a weight that becomes a second-order operator at the terminal time. Define
	\begin{equation}
		\label{eq:higher-weight}
		G_s (t) = \mathcal U(t) \op \left( \left( \left\langle \frac{\eta}{k} \right\rangle + 2U'_{\sup}(s-t) \right)^2 \right) \mathcal U(t)^{-1}.
	\end{equation}
	
	When differentiating the \(G_s\)-weighted energy, we distinguish whether the time derivative falls on \(G_s\) or on the remaining factors. Since \(G_s\) commutes with \(D_t\), \(\Delta_t\), and \(S\), the terms in the latter class have the same structure as the corresponding terms in Proposition \ref{prop:basic-energy}, apart from commutator errors involving \(G_s\) and \(U'\).
	
	The terms with \(\partial_t G_s\) supply an additional negative contribution. To see this formally, write
	\[G_s = \left( \left\langle \frac{D_t}{ik} \right\rangle + 2U'_{\sup}(s-t) \right)^2.\]
	At the level of principal symbols, differentiation gives	
	\[(\partial_t G_s) G_s^{-1} \approx  -2 \left( 2U'_{\sup} + U'\frac{D_t}{ik} \left\langle \frac{D_t}{ik} \right\rangle^{-1} \right) \left( \left\langle \frac{D_t}{ik} \right\rangle + 2U'_{\sup}(s-t) \right)^{-1}.\]
	We choose the translation \(2U'_{\sup}(s-t)\) so that the leading term in the time derivative of \(G_s\) is negative. Under the slow-variation assumption, the negative principal part of this contribution absorbs its remainder terms. At \(t = s\), we have \(G_s (s) = - k^{-2} \Delta_s\). Thus comparing the weighted energies at \(t = 0\) and \(t = s\) then yields the higher-order estimate.
	
	\begin{proposition}[Higher-order energy estimate]
		\label{prop:higher-energy}
		Under the assumptions on \(U\) and \(\beta\) in Theorem \ref{thm:stability}, let \((V,B)\) solve \eqref{eq:transformed-linear}. Define
		\[\mathcal E_{2,s} = \frac{1}{2} \| S G_s \widetilde V \|_{L^2}^2 + \frac{1}{2} \| S G_s B \|_{L^2}^2.\]
		Then
		\begin{equation}
			\label{eq:higher-energy-equivalence}
			\mathcal E_{2,s} \approx \| G_s V \|_{L^2}^2 + \| G_s B \|_{L^2}^2,
		\end{equation}
		and
		\begin{equation}
			\label{eq:higher-energy-decay}
			\frac{d}{dt} \mathcal E_{2,s} (t) \le 0, \qquad 0 \le t \le s.
		\end{equation}
		Consequently,
		\begin{equation}
			\label{eq:higher-endpoint}
			\| (-\Delta_s) (V,B) \|_{L^2} (s) \lesssim \| (-\Delta_k) (V,B) \|_{L^2} (0) \langle s \rangle^2.
		\end{equation}
	\end{proposition}
	
	\begin{proof}
		\noindent\textbf{Step 1. Energy equivalence.} We first prove \eqref{eq:higher-energy-equivalence}. Applying Lemma \ref{lem:shear-conjugation} to \(G_s\), we obtain
		\[\left\| G_s U' G_s^{-1} - U' \right\|_{L^2 \to L^2} \lesssim \epsilon,\]
		uniformly in \(s \ge 0\) and \(0 \le t \le s\). The definition of \(\widetilde V\) gives
		\[\| G_s (\widetilde V - V) \|_{L^2} = \frac{1}{|\beta k|} \| G_s (U' B_2) \|_{L^2} \le \frac{U'_{\sup} + C \epsilon}{|\beta k|} \| G_s B_2 \|_{L^2} \lesssim \| G_s B \|_{L^2}.\]
		The boundedness of \(S\) and \(S^{-1}\), followed by the triangle inequality in both directions, proves \eqref{eq:higher-energy-equivalence}.
		
		\noindent\textbf{Step 2. Energy derivative and the basic terms.} We next compute the time derivative of \(\mathcal E_{2,s}\). The terms involving \(ik \beta\) in \eqref{eq:corrected-system} cancel after taking real parts. Hence, we are left with the weight derivatives and the pressure term. Since \(S G_s\) commutes with \(\nabla_t\), integration by parts gives
		\[- \re \langle S G_s \widetilde V, S G_s \nabla_t P \rangle = \frac{1}{\beta} \re \langle S G_s (U' B_2), S G_s P \rangle.\]
		Substituting \(P = 2ik (-\Delta_t)^{-1} (U'\widetilde V_2)\) and differentiating the weights, we obtain
		\begin{align*}
			\frac{d}{dt}\mathcal E_{2,s} =& \re \langle S G_s \widetilde V, (\partial_t S) G_s \widetilde V \rangle + \re \langle S G_s B, (\partial_t S) G_s B\rangle \\
			&+ \re \langle S G_s \widetilde V, S (\partial_t G_s) \widetilde V \rangle + \re \langle S G_s B, S (\partial_t G_s) B \rangle \\
			&+ \frac{2}{\beta} \re \langle S G_s (U'B_2), ik (-\Delta_t)^{-1} S G_s (U' \widetilde V_2) \rangle.
		\end{align*}
		
		For the terms involving \((\partial_tS)S^{-1}\), the same argument as in the proof of Proposition \ref{prop:basic-energy} gives
		\[\re \langle S G_s \widetilde V, (\partial_t S) G_s \widetilde V \rangle \le (- \gamma U'_{\inf} + C\epsilon) \left\| k (-\Delta_t)^{-1/2} S G_s \widetilde V \right\|_{L^2}^2,\]
		and
		\[\re \langle S G_s B, (\partial_t S) G_s B \rangle \le (- \gamma U'_{\inf} + C\epsilon) \left\| k (-\Delta_t)^{-1/2} S G_s B \right\|_{L^2}^2.\]
		
		We now estimate the pressure term. The symbol of \((-\Delta_t)^{-1/2} S G_s\) satisfies the derivative and shift assumptions of Lemma \ref{lem:shear-conjugation} uniformly in \(s \ge 0\) and \(0 \le t \le s\). Hence
		\[\left\| (-\Delta_t)^{-1/2} S G_s U' G_s^{-1} S^{-1} (-\Delta_t)^{1/2} - U' \right\|_{L^2 \to L^2} \lesssim \epsilon.\]
		The pressure-term calculation in the proof of Proposition \ref{prop:basic-energy} therefore carries over with \(S\) replaced by \(SG_s\), giving
		\[\left| \frac{2}{\beta} \re \langle S G_s (U' B_2), ik (-\Delta_t)^{-1} S G_s (U' \widetilde V_2) \rangle \right| \le \left( \frac{(U'_{\sup})^2}{|\beta|} + C\epsilon \right) \left( \left\| k (-\Delta_t)^{-1/2} S G_s B \right\|_{L^2}^2 + \left\| k (-\Delta_t)^{-1/2} S G_s \widetilde V \right\|_{L^2}^2 \right).\]
		Together with the estimates for the two terms involving \(\partial_t S\), this shows that the total contribution of these three terms to \(\frac{d}{dt}\mathcal E_{2,s}\) is nonpositive under the smallness assumption on \(\epsilon\). It remains to estimate the terms involving \(\partial_t G_s\).
		
		\noindent\textbf{Step 3. The terms involving \(\partial_t G_s\).} Differentiating \eqref{eq:higher-weight} gives
		\[\partial_t G_s = ik [U, G_s] - 4 U'_{\sup} \mathcal U \op \left( \left\langle \frac{\eta}{k} \right\rangle + 2U'_{\sup}(s-t) \right) \mathcal U^{-1}.\]
		Applying the composition formula \eqref{eq:composition_formula} and expanding the multiplier to first order in the oscillatory variable, as in the proof of Lemma \ref{lem:weight-derivative}, we obtain
		\begin{equation}
			\label{eq:higher-weight-derivative}
			\sigma \left( \mathcal U^{-1} (\partial_t G_s) G_s^{-1} \mathcal U \right) (y,\eta) = -2 \left( 2U'_{\sup} + U'(y) \frac{\eta}{k} \left\langle \frac{\eta}{k} \right\rangle^{-1} \right) \left( \left\langle \frac{\eta}{k} \right\rangle + 2U'_{\sup}(s-t) \right)^{-1} + r_G(y,\eta),
		\end{equation}
		where
		\begin{equation}
			\label{eq:higher-remainder}
			r_G(y,\eta) = \frac{ik}{2\pi} \int_{\mathbb R} \int_{\mathbb R} \int_{0}^{1} e^{-iz\lambda}U''(y+z) (1-\theta) \partial_\eta^2 \left( \left\langle \frac{\eta + \theta \lambda}{k} \right\rangle + 2U'_{\sup}(s-t) \right)^2 \left( \left\langle \frac{\eta}{k} \right\rangle + 2U'_{\sup}(s-t) \right)^{-2} d\theta d\lambda dz.
		\end{equation}
		For convenience, we introduce the auxiliary self-adjoint operator
		\[P_s = \mathcal U \op \left( \left( \left\langle \frac{\eta}{k} \right\rangle + 2U'_{\sup}(s-t) \right)^{-1/2} \right) \mathcal U^{-1}.\]
		Then the principal term in \eqref{eq:higher-weight-derivative} is of the order of \(P_s^2\).
		
		To estimate \(r_G\), direct differentiation gives, for \(m \in \{0, 1, 2\}\),
		\[\left| \partial_\lambda^m \partial_\eta^2 \left( \left\langle \frac{\eta + \theta \lambda}{k} \right\rangle + 2U'_{\sup}(s-t) \right)^2 \right| \lesssim |k|^{-1} \left( \left\langle \frac{\eta + \theta\lambda}{k} \right\rangle + 2U'_{\sup}(s-t) \right).
		\]
		Peetre's inequality implies
		\[\left( \left\langle \frac{\eta + \theta\lambda}{k} \right\rangle + 2U'_{\sup}(s-t) \right) \left( \left\langle \frac{\eta}{k} \right\rangle + 2U'_{\sup}(s-t) \right)^{-1} \lesssim \langle \lambda \rangle.\]
		Integrating by parts in \eqref{eq:higher-remainder}, using \(1 - \partial_\lambda^2\) to gain \(\langle z \rangle^{-2}\) and \((1 - \partial_z^2)^2\) to gain \(\langle \lambda \rangle^{-4}\), we obtain
		\[|r_G(y,\eta)| \lesssim \epsilon \int_{\mathbb R} \int_{\mathbb R} \langle z \rangle^{-2} \langle \lambda \rangle^{-4} \langle \lambda \rangle \left( \left\langle \frac{\eta}{k} \right\rangle + 2U'_{\sup}(s-t) \right)^{-1} d\lambda dz
		\lesssim \epsilon \left( \left\langle \frac{\eta}{k} \right\rangle + 2U'_{\sup}(s-t) \right)^{-1}.\]
		The same argument applied after differentiating \eqref{eq:higher-remainder} yields
		\[\left| \partial_y^j \partial_\eta^l r_G(y,\eta) \right| \lesssim \epsilon \left( \left\langle \frac{\eta}{k} \right\rangle + 2U'_{\sup}(s-t) \right)^{-1}, \qquad j, l \in \{0, \cdots, 5\}.\]
		
		We denote the remainder operator by \(\mathcal R_G = \mathcal U \op(r_G) \mathcal U^{-1}\). We now estimate \(\mathcal R_G\) relative to \(P_s^2\). By the composition formula \eqref{eq:composition_formula},
		\begin{align*}
			\sigma \left( \mathcal U^{-1} P_s^{-1} \mathcal R_G P_s^{-1} \mathcal U \right) (y, \eta) =& \left( \left\langle \frac{\eta}{k} \right\rangle + 2U'_{\sup}(s-t) \right)^{1/2} \# r_G(y,\eta) \# \left( \left\langle \frac{\eta}{k} \right\rangle + 2U'_{\sup}(s-t) \right)^{1/2} \\
			=& \frac{1}{2\pi} \int_{\mathbb R} \int_{\mathbb R} e^{-iz\lambda} \left( \left\langle \frac{\eta + \lambda}{k} \right\rangle + 2U'_{\sup}(s-t) \right)^{1/2} \\
			&\times r_G(y+z,\eta) \left( \left\langle \frac{\eta}{k} \right\rangle + 2U'_{\sup}(s-t)\right)^{1/2} d\lambda dz.
		\end{align*}
		Combining the preceding bounds for \(r_G\) with the same shift estimate and integration-by-parts argument, we obtain
		\[\left| \partial_y^j \partial_\eta^l \sigma \left( \mathcal U^{-1} P_s^{-1} \mathcal R_G P_s^{-1} \mathcal U \right) (y,\eta) \right| \lesssim \epsilon, \qquad j, l \in \{0, 1\}.\]
		Since \(\mathcal U\) is unitary, the Calder\'on--Vaillancourt Theorem \ref{thm:CV} yields
		\begin{equation}
			\label{eq:higher-remainder-bound}
			\left\| P_s^{-1} \mathcal R_G P_s^{-1} \right\|_{L^2 \to L^2} \lesssim \epsilon.
		\end{equation}
		
		We now use \eqref{eq:higher-remainder-bound} to estimate the contribution of \(\partial_t G_s\) to the energy derivative. We rewrite the term containing \(\widetilde V\) in the \(P_s\)-weighted form
		\[\re \langle S G_s \widetilde V, S (\partial_t G_s) \widetilde V \rangle = \re \left\langle P_s S G_s \widetilde V, \left( P_s^{-1} S (\partial_t G_s) G_s^{-1} S^{-1} P_s^{-1} \right) P_s S G_s \widetilde V \right\rangle.\]
		Substituting the decomposition \eqref{eq:higher-weight-derivative} into the operator in this quadratic form and using the definitions of \(P_s\) and \(\mathcal R_G\), we obtain
		\[P_s^{-1} S(\partial_t G_s) G_s^{-1} S^{-1} P_s^{-1} = -4U'_{\sup} -2 \left( S P_s^{-1} U' (S P_s^{-1})^{-1} \right) \mathcal U \op \left(\frac{\eta}{k} \left\langle \frac{\eta}{k} \right\rangle^{-1} \right) \mathcal U^{-1} + S P_s^{-1} \mathcal R_G P_s^{-1} S^{-1}.\]
		Here we used that \(S, G_s, P_s\), and \(\mathcal U \op \left( \frac{\eta}{k} \left\langle \frac{\eta}{k} \right\rangle^{-1} \right) \mathcal U^{-1}\) commute with one another. Applying Lemma \ref{lem:shear-conjugation} to \(S P_s^{-1}\), we obtain
		\[\left\| S P_s^{-1} U' (S P_s^{-1})^{-1} - U' \right\|_{L^2 \to L^2} \lesssim \epsilon.\]
		The multiplier \(\frac{\eta}{k} \left\langle \frac{\eta}{k} \right\rangle^{-1}\) has absolute value at most one. Together with \(\| U' \|_{L^2 \to L^2}\le U'_{\sup}\), this gives
		\[\left\| -2 \left( S P_s^{-1} U' (S P_s^{-1})^{-1} \right) \mathcal U \op \left( \frac{\eta}{k} \left\langle \frac{\eta}{k} \right\rangle^{-1} \right) \mathcal U^{-1} \right\|_{L^2 \to L^2} \le 2U'_{\sup} + C\epsilon.\]
		Moreover, the boundedness of \(S\) and \(S^{-1}\), together with \eqref{eq:higher-remainder-bound}, gives
		\[\left\| S P_s^{-1} \mathcal R_G P_s^{-1} S^{-1} \right\|_{L^2 \to L^2} \lesssim \epsilon.\]
		Thus
		\[\re \langle S G_s \widetilde V, S (\partial_t G_s) \widetilde V \rangle \le \left(-2U'_{\sup} + C \epsilon \right) \| P_s S G_s \widetilde V \|_{L^2}^2.\]
		The same estimate holds with \(\widetilde V\) replaced by \(B\). Indeed, we have
		\[\re \langle S G_s B, S (\partial_t G_s) B \rangle \le \left(-2U'_{\sup} + C \epsilon \right) \| P_s S G_s B \|_{L^2}^2.\]
		
		\noindent\textbf{Step 4. Completion of the estimate.} Combining the preceding estimates and recalling the definition of \(\gamma\), we obtain
		\begin{align*}
			\frac{d}{dt} \mathcal E_{2,s} \le& \left( -U'_{\inf} + C\epsilon \right) \left( \left\| k (-\Delta_t)^{-1/2} S G_s \widetilde V \right\|_{L^2}^2 + \left\| k (-\Delta_t)^{-1/2} S G_s B \right\|_{L^2}^2 \right) \\
			&+ \left( -2U'_{\sup} + C\epsilon \right) \left( \| P_s S G_s \widetilde V \|_{L^2}^2 + \| P_s S G_s B \|_{L^2}^2 \right) \\
			\le& 0.
		\end{align*}
		The final inequality follows from the assumed smallness of \(\epsilon\), proving \eqref{eq:higher-energy-decay}. Together with \eqref{eq:higher-energy-equivalence},
		\[\| G_s (V,B) \|_{L^2}(s) \lesssim \| G_s (V,B) \|_{L^2} (0).\]
		At \(t = s\), \(G_s(s) = -k^{-2} \Delta_s\), while at \(t = 0\),
		\[\sigma (G_s(0)) (y,\eta) = \left( \left\langle \frac{\eta}{k} \right\rangle + 2U'_{\sup}s \right)^2 \lesssim \left\langle \frac{\eta}{k} \right\rangle^2 \langle s \rangle^2,\]
		since \(\mathcal U(0) = 1\). Therefore, we obtain \eqref{eq:higher-endpoint}, which completes the proof of Proposition \ref{prop:higher-energy}.
	\end{proof}
	
	\subsection{The lower-order energy estimate}
	
	For the lower-order estimate, we choose
	\[h(\xi) = 1 + \xi \langle \xi \rangle^{-1},\]
	and define
	\begin{equation}
		\label{eq:lower-weight}
		H_s (t) = \mathcal U (t) \op \left( h \left( \frac{\eta}{k} + \frac{U'_{\inf}}{2}(t-s) \right) \right) \mathcal U^{-1} (t).
	\end{equation}
	
	Formally, at the level of principal symbols,
	\[(\partial_t H_s) H_s^{-1} \approx \left( \frac{U'_{\inf}}{2} - U'(y) \right) h' \left( \frac{D_t}{ik} + \frac{U'_{\inf}}{2}(t-s) \right) h \left( \frac{D_t}{ik} + \frac{U'_{\inf}}{2}(t-s) \right)^{-1}.\]
	The leading symbol is negative, since \(h'(\xi) = \langle \xi \rangle^{-3} > 0\). The translation here is not used to adjust the sign, which already holds without it, but rather to produce decay in the terminal time. Indeed, at \(t = 0\),
	\[h \left( \frac{\eta}{k} - \frac{U'_{\inf}}{2} s \right) \lesssim \left\langle \frac{\eta}{k} \right\rangle^2 \langle s \rangle^{-2},\]
	while at \(t = s\),
	\[h \left( \frac{\eta}{k} \right) \gtrsim \left\langle \frac{\eta}{k} \right\rangle^{-2}.\]
	Comparing the weighted energies at these two times gives the lower-order decay estimate.
	
	\begin{proposition}[Lower-order energy estimate]
		\label{prop:lower-energy}
		Under the assumptions on \(U\) and \(\beta\) in Theorem \ref{thm:stability}, let \((V,B)\) solve \eqref{eq:transformed-linear}. Define
		\[\mathcal E_{-2,s} = \frac{1}{2} \| S H_s \widetilde V \|_{L^2}^2 + \frac{1}{2} \| S H_s B \|_{L^2}^2.\]
		Then
		\begin{equation}
			\label{eq:lower-energy-equivalence}
			\mathcal E_{-2,s} \approx \| H_s V \|_{L^2}^2 + \| H_s B \|_{L^2}^2,
		\end{equation}
		and
		\begin{equation}
			\label{eq:lower-energy-decay}
			\frac{d}{dt} \mathcal E_{-2,s} (t) \le 0, \qquad 0 \le t \le s.
		\end{equation}
		Consequently,
		\begin{equation}
			\label{eq:lower-endpoint}
			\| k^2 (-\Delta_s)^{-1} (V,B) \|_{L^2} (s) \lesssim  \| k^{-2} (-\Delta_k) (V,B) \|_{L^2} (0) \langle s \rangle^{-2}.
		\end{equation}
	\end{proposition}
	
	\begin{proof}
		Direct differentiation and Peetre's inequality give, for each fixed integer \(m \ge 1\),
		\begin{equation}
			\label{eq:h-properties}
			|h^{(m)}(\xi)| \lesssim h'(\xi), \qquad \frac{h(\xi + \lambda)}{h(\xi)} \lesssim \langle \lambda \rangle^2, \qquad \frac{h'(\xi + \lambda)}{h'(\xi)} \lesssim \langle \lambda \rangle^3, \qquad \frac{(h'/h)(\xi + \lambda)}{(h'/h)(\xi)} \lesssim \langle \lambda \rangle^4.
		\end{equation}
		By \eqref{eq:h-properties}, the arguments in the proof of Proposition \ref{prop:higher-energy} imply \eqref{eq:lower-energy-equivalence}. They also show that all contributions to \(\frac{d}{dt} \mathcal E_{-2,s}\) except those involving \(\partial_t H_s\) are nonpositive when combined. It remains to estimate the \(\partial_t H_s\) terms.
		
		Differentiating \eqref{eq:lower-weight} and applying the same first-order symbolic expansion as in Lemma \ref{lem:weight-derivative}, we obtain
		\begin{equation}
			\label{eq:lower-weight-derivative}
			\sigma \left( \mathcal U^{-1} (\partial_t H_s) H_s^{-1} \mathcal U \right) (y,\eta) = - \left( U'(y) - \frac{U'_{\inf}}{2} \right) h' \left( \frac{\eta}{k} + \frac{U'_{\inf}}{2}(t-s) \right) h \left( \frac{\eta}{k} + \frac{U'_{\inf}}{2}(t-s) \right)^{-1} + r_H(y,\eta),
		\end{equation}
		where the integral remainder is
		\[r_H(y,\eta) = \frac{ik}{2\pi} \int_{\mathbb R} \int_{\mathbb R} \int_{0}^{1} e^{-iz\lambda} U''(y+z) (1-\theta) \frac{1}{k^2} h'' \left( \frac{\eta + \theta \lambda}{k} + \frac{U'_{\inf}}{2} (t-s) \right) h \left( \frac{\eta}{k} + \frac{U'_{\inf}}{2}(t-s) \right)^{-1} d\theta d\lambda dz.\]
		Using \eqref{eq:h-properties} and integrating by parts in the oscillatory variables, we obtain
		\[\left| \partial_y^j \partial_\eta^l r_H(y,\eta) \right| \lesssim\epsilon h' \left( \frac{\eta}{k} + \frac{U'_{\inf}}{2}(t-s) \right) h \left(\frac{\eta}{k} + \frac{U'_{\inf}}{2}(t-s) \right)^{-1}, \qquad j, l \in \{0, \cdots, 5\}.\]
		
		We define the auxiliary weight
		\[Q_s = \mathcal U \op \left( h' \left( \frac{\eta}{k} + \frac{U'_{\inf}}{2} (t-s) \right)^{1/2} h \left( \frac{\eta}{k} + \frac{U'_{\inf}}{2}(t-s) \right)^{-1/2} \right) \mathcal U^{-1},\]
		and denote the remainder operator by \(\mathcal R_H = \mathcal U \op(r_H) \mathcal U^{-1}\). The bounds for \(r_H\) above, together with \eqref{eq:h-properties}, the composition formula \eqref{eq:composition_formula}, and the Calder\'on--Vaillancourt Theorem \ref{thm:CV}, yield
		\begin{equation}
			\label{eq:lower-remainder-bound}
			\left\| Q_s^{-1} \mathcal R_H Q_s^{-1} \right\|_{L^2 \to L^2} \lesssim \epsilon.
		\end{equation}
				
		As in the higher-order proof, we write the terms involving \(\partial_t H_s\) in the \(Q_s\)-weighted form. By \eqref{eq:lower-weight-derivative} and the commutativity of \(S\), \(H_s\), and \(Q_s\), the operator in the resulting quadratic forms has the exact decomposition
		\[Q_s^{-1} S (\partial_t H_s) H_s^{-1} S^{-1} Q_s^{-1} = \frac{U'_{\inf}}{2} - (S Q_s^{-1}) U' (S Q_s^{-1})^{-1} + S Q_s^{-1} \mathcal R_H Q_s^{-1} S^{-1}.\]
		Lemma \ref{lem:shear-conjugation} shows that the second term differs from \(-U'\) by at most \(C\epsilon\) in \(L^2 \to L^2\) norm. The remainder bound \eqref{eq:lower-remainder-bound}, together with the boundedness of \(S\) and \(S^{-1}\), controls the last term. Since \(U' \ge U'_{\inf}\), we conclude that
		\[\re \langle S H_s \widetilde V, S (\partial_t H_s) \widetilde V \rangle + \re \langle S H_s B, S (\partial_t H_s) B \rangle \le \left(- \frac{U'_{\inf}}{2} + C\epsilon\right) \left( \| Q_s S H_s \widetilde V \|_{L^2}^2 + \| Q_s S H_s B \|_{L^2}^2 \right).\]
		
		By combining the estimates, \eqref{eq:lower-energy-decay} follows. Finally, the terminal and initial multiplier bounds stated above imply
		\[\| k^2 (-\Delta_s)^{-1} (V,B) \|_{L^2} (s) \lesssim \mathcal E_{-2,s} (s)^{1/2} \le \mathcal E_{-2,s} (0)^{1/2} \lesssim  \| k^{-2} (-\Delta_k) (V,B) \|_{L^2} (0) \langle s\rangle^{-2},\]
		which finishes the proof.
	\end{proof}
	
	\subsection{Proof of Theorem \ref{thm:stability}}
	
	We now prove stability and inviscid damping by combining the energy estimates above. Since we will sum over nonzero horizontal Fourier modes, we restore the subscript \(k\) on the transformed variables and write \(\mathcal U_k\) and \(\Delta_{t,k}\) for the corresponding mode-dependent operators.
	
	\begin{proof}[Proof of Theorem \ref{thm:stability}]
		\noindent\textbf{Step 1. Upper bounds.} By \eqref{eq:potential-representation} and integration by parts, we have
		\[\| (V_{2,k},B_{2,k}) \|_{L^2}^2 (t) = \| k (\Psi_k,\Phi_k) \|_{L^2}^2 (t) = \| k (-\Delta_{t,k})^{-1/2} (V_k,B_k) \|_{L^2}^2 (t) = \left\langle (V_k,B_k), k^2 (-\Delta_{t,k})^{-1} (V_k,B_k) \right\rangle (t),\]
		and
		\[\| (\Omega_k,J_k) \|_{L^2}^2 (t) = \| (-\Delta_{t,k}) (\Psi_k,\Phi_k) \|_{L^2}^2 (t) = \| (-\Delta_{t,k})^{1/2} (V_k,B_k) \|_{L^2}^2 (t) = \langle (V_k,B_k), (-\Delta_{t,k}) (V_k,B_k) \rangle (t).\]
		Combining the Cauchy--Schwarz inequality with \eqref{eq:mode-uniform-bound}, \eqref{eq:higher-endpoint}, and \eqref{eq:lower-endpoint}, we obtain
		\begin{align*}
			\| (V_{1,k},B_{1,k}) \|_{L^2} (t) &\lesssim \| (V_k,B_k) \|_{L^2} (0), \\
			\| (V_{2,k},B_{2,k}) \|_{L^2}^2 (t) &\lesssim \| (V_k,B_k) \|_{L^2} (0) \| k^{-2} (-\Delta_k) (V_k,B_k) \|_{L^2} (0) \langle t \rangle^{-2}, \\
			\| (\Omega_k,J_k) \|_{L^2}^2 (t) &\lesssim \| (V_k,B_k) \|_{L^2} (0) \| (-\Delta_k)(V_k,B_k) \|_{L^2} (0) \langle t \rangle^2.
		\end{align*}
		Since \(\mathcal U_k(t)\) is unitary and \((V_k,B_k)(0) = (v_{\mathrm{in},k},b_{\mathrm{in},k})\), summing the mode estimates over \(k \neq 0\) and applying Plancherel's theorem yields
		\begin{align*}
			\| (v_1,b_1) \|_{L_x^2L_y^2} (t) &\lesssim \| (v_{\mathrm{in}},b_{\mathrm{in}}) \|_{L_x^2L_y^2}, \\
			\| (v_2,b_2) \|_{L_x^2L_y^2} (t) &\lesssim \| (v_{\mathrm{in}},b_{\mathrm{in}}) \|_{L_x^2L_y^2}^{1/2} \| (v_{\mathrm{in}},b_{\mathrm{in}}) \|_{H_{x,y}^2}^{1/2} \langle t \rangle^{-1}, \\
			\| (\omega,j) \|_{L_x^2L_y^2}(t) &\lesssim \| (v_{\mathrm{in}},b_{\mathrm{in}}) \|_{L_x^2L_y^2}^{1/2} \| (v_{\mathrm{in}},b_{\mathrm{in}}) \|_{H_{x,y}^2}^{1/2} \langle t \rangle.
		\end{align*}
		This proves the upper bounds in Theorem \ref{thm:stability}.
		
		\noindent\textbf{Step 2. Lower bounds.} Fix \(s \ge 0\) and, for \(0 \le t \le s\), set
		\[v^\sharp(t,x,y) = (v_1,-v_2) (s-t,-x,y), \qquad b^\sharp(t,x,y) = (b_1,-b_2) (s-t,-x,y), \qquad p^\sharp(t,x,y) = p(s-t,-x,y).\]
		A direct substitution shows that \((v^\sharp,b^\sharp,p^\sharp)\) solves \eqref{eq:linear} with the same \(U\) and \(\beta\). Applying the \(L_x^2L_y^2\) bound established in Step 1 to the solution \((v^\sharp,b^\sharp)\), we obtain
		\begin{equation}
			\label{eq:uniform-reverse-energy-bound}
			\| (v_{\mathrm{in}},b_{\mathrm{in}}) \|_{L_x^2L_y^2} = \| (v^\sharp,b^\sharp) \|_{L_x^2L_y^2} (s) \lesssim \| (v^\sharp,b^\sharp) \|_{L_x^2L_y^2}(0) = \| (v,b) \|_{L_x^2L_y^2} (s).
		\end{equation}
		Since \(s \ge 0\) is arbitrary and \((v_{\mathrm{in}},b_{\mathrm{in}}) \neq 0\), there exists \(c_{\mathrm{in}} > 0\), depending on \(U'_{\inf},U'_{\sup},\beta\) and the initial perturbation, such that
		\begin{equation}
			\label{eq:reverse-energy-bound}
			\| (v,b) \|_{L_x^2L_y^2} (t) \ge c_{\mathrm{in}}, \qquad t \ge 0.
		\end{equation}
		
		We first consider the horizontal components. By \eqref{eq:reverse-energy-bound} and the upper bound for the vertical components from Step 1,
		\[\| (v_1,b_1) \|_{L_x^2L_y^2}^2 (t) = \| (v,b) \|_{L_x^2L_y^2}^2 (t) - \| (v_2,b_2) \|_{L_x^2L_y^2}^2 (t)\ge c_{\mathrm{in}}^2 - C_{\mathrm{in}}^2 \langle t \rangle^{-2},\]
		where \(C_{\mathrm{in}} > 0\) is the constant in the vertical-component upper bound. Thus, for some sufficiently large \(t^*>0\), the horizontal norm is uniformly bounded away from zero for \(t \ge t^*\). It remains to show that \(\| (v_1,b_1) \|_{L_x^2L_y^2}\) is strictly positive on \([0,t^*]\). We argue by contradiction. If it vanished at some \(t_0 \in [0,t^*]\), the divergence-free conditions would give
		\[\partial_y v_2 (t_0) = - \partial_x v_1 (t_0) = 0, \qquad \partial_y b_2 (t_0) = - \partial_x b_1 (t_0) = 0.\]
		But square integrability would force \(v_2 (t_0) = b_2 (t_0) = 0\), contradicting \eqref{eq:reverse-energy-bound}.
		
		We now turn to the vertical components and the vorticity--current pair. On the nonzero horizontal modes, the Biot--Savart formula gives
		\[v = \nabla^\perp \Delta^{-1} \omega, \qquad b = \nabla^\perp\Delta^{-1}j.\]
		Applying the Cauchy--Schwarz inequality in Fourier variables and using \(|k|\ge1\), we obtain
		\begin{equation}
			\label{eq:lower-interpolation}
			\| (v,b) \|_{L_x^2L_y^2}^2 (t) = \| (-\Delta)^{-1/2} (\omega,j) \|_{L_x^2L_y^2}^2 (t) = \langle (-\Delta)^{-1} (\omega,j), (\omega,j) \rangle (t) \le \| (v_2,b_2) \|_{L_x^2L_y^2} (t) \|(\omega,j) \|_{L_x^2L_y^2} (t).
		\end{equation}
		Combining \eqref{eq:reverse-energy-bound} and \eqref{eq:lower-interpolation} with the upper bounds from Step 1, we obtain, after decreasing \(c_{\mathrm{in}}\) if necessary,
		\[\| (v_2,b_2) \|_{L_x^2L_y^2} (t) \ge c_{\mathrm{in}} \langle t \rangle^{-1}, \qquad \| (\omega,j) \|_{L_x^2L_y^2} (t) \ge c_{\mathrm{in}} \langle t \rangle.\]
		Together with the horizontal lower bound, these estimates complete the proof of Theorem \ref{thm:stability}.
	\end{proof}
	
	\section{Scattering and leading asymptotics}
	\label{sec4}
	
	We now turn to the scattering problem. We also work with a fixed \(k \neq 0\) and suppress the mode index on the transformed variables and operators until we sum over the nonzero modes.
	
	\subsection{The scattering system and preparatory estimates}
	
	Recalling the definition of \(F_\pm\) in \eqref{eq:phase-corrected-amplitudes}, we first derive the equations satisfied by \(\omega\) and \(j\). Taking the curl of the two evolution equations in \eqref{eq:linear} and using the divergence-free conditions, we obtain
	\[
	\begin{dcases}
		(\partial_t + U(y) \partial_x) \omega = \beta \partial_x j + U''(y) v_2,\\
		(\partial_t + U(y) \partial_x) j = \beta \partial_x \omega - 2 U'(y) \partial_y b_2 - U''(y) b_2.
	\end{dcases}
	\]
	After taking the Fourier transform in \(x\), we fix a nonzero horizontal mode \(k\) and apply the shear transformation \eqref{eq:shear-transform} to obtain
	\[
	\begin{dcases}
		\partial_t \Omega = ik \beta J + U''V_2,\\
		\partial_t J = ik \beta \Omega - 2 U' D_t B_2 - U'' B_2.
	\end{dcases}
	\]
	Using \eqref{eq:potential-representation}, we express \(V\) and \(B\) in terms of \(\Omega\) and \(J\). Thus
	\begin{equation}
		\label{eq:scattering-vorticity-current}
		\begin{dcases}
			\partial_t \Omega = ik \beta J + ik U'' \Delta_t^{-1} \Omega,\\
			\partial_t J = ik \beta \Omega + (\partial_t \Delta_t) \Delta_t^{-1} J,
		\end{dcases}
	\end{equation}
	where it is important to note that
	\begin{equation}
		\label{eq:laplacian-time-derivative}
		\partial_t \Delta_t = -ik (U' D_t + D_t U' ) = -ik (2 D_t U' - U'') = -ik (2U' D_t + U'').
	\end{equation}
	We also express the amplitudes \(F_\pm\) in sheared coordinates. Specifically, \eqref{eq:phase-corrected-amplitudes} and \eqref{eq:shear-transform} imply
	\begin{equation}
		\label{eq:sheared-amplitudes}
		F_\pm = e^{\mp ik \beta t} (-\Delta_t)^{-1/2} (\Omega \pm J).
	\end{equation}
	
	We now differentiate the normalized vorticity and current density appearing in \(F_\pm\). By \eqref{eq:scattering-vorticity-current}, we have
	\[\partial_t \left( (-\Delta_t)^{-1/2} \Omega \right) = \left( \partial_t (-\Delta_t)^{-1/2} \right) \Omega + ik \beta (-\Delta_t)^{-1/2} J + ik (-\Delta_t)^{-1/2} U'' \Delta_t^{-1} \Omega,\]
	\[\partial_t \left( (-\Delta_t)^{-1/2} J \right) = \left( \partial_t (-\Delta_t)^{-1/2} \right) J + ik \beta (-\Delta_t)^{-1/2} \Omega + (-\Delta_t)^{-1/2} (\partial_t \Delta_t) \Delta_t^{-1} J.\]
	To combine the terms arising from the time dependence of \(\Delta_t\) in the second equation, we note that differentiating the identity \((-\Delta_t)^{-1/2} (-\Delta_t) (-\Delta_t)^{-1/2} = I\) with respect to time yields
	\[(-\Delta_t)^{-1/2} (\partial_t \Delta_t) (-\Delta_t)^{-1/2} = \left( \partial_t (-\Delta_t)^{-1/2} \right) (-\Delta_t)^{1/2} + (-\Delta_t)^{1/2} \left( \partial_t (-\Delta_t)^{-1/2}\right).\]
	Thus we can rewrite both equations in terms of the normalized variables as follows:
	\begin{align*}
		\partial_t \left( (-\Delta_t)^{-1/2} \Omega \right) =& ik \beta (-\Delta_t)^{-1/2} J + \left( \partial_t (-\Delta_t)^{-1/2} \right) (-\Delta_t)^{1/2} \left( (-\Delta_t)^{-1/2} \Omega \right) \\
		&- ik (-\Delta_t)^{-1/2} U'' (-\Delta_t)^{-1/2} \left( (-\Delta_t)^{-1/2} \Omega \right),
	\end{align*}
	\[\partial_t \left( (-\Delta_t)^{-1/2} J \right) = ik \beta (-\Delta_t)^{-1/2} \Omega - (-\Delta_t)^{1/2} \left( \partial_t (-\Delta_t)^{-1/2} \right) \left( (-\Delta_t)^{-1/2} J \right).\]
	Using the normalized evolution equations above, we compute \(\partial_t F_\pm\). The terms involving \(\partial_t (-\Delta_t)^{-1/2}\) give one operator acting on \(F_\pm\) and another acting on \(e^{\mp 2ik\beta t}F_\mp\). We denote these operators by \(K\) and \(R\), respectively:
	\begin{equation}
		\label{eq:scattering-K}
		K = \frac{1}{2} \left( \left(\partial_t (-\Delta_t)^{-1/2} \right) (-\Delta_t)^{1/2} - (-\Delta_t)^{1/2} \left( \partial_t (-\Delta_t)^{-1/2} \right) \right),
	\end{equation}
	\begin{equation}
		\label{eq:scattering-R}
		R = \frac{1}{2} \left( \left( \partial_t (-\Delta_t)^{-1/2} \right) (-\Delta_t)^{1/2} + (-\Delta_t)^{1/2} \left( \partial_t (-\Delta_t)^{-1/2} \right) \right) = \frac{1}{2} (-\Delta_t)^{-1/2} (\partial_t \Delta_t) (-\Delta_t)^{-1/2}.
	\end{equation}
	When we differentiate \eqref{eq:sheared-amplitudes}, the derivatives of the phase factors cancel the terms proportional to \(ik\beta\). Thus
	\begin{equation}
		\label{eq:scattering-system}
		\partial_t F_\pm = K F_\pm + e^{\mp 2ik\beta t} R F_\mp - \frac{ik}{2} (-\Delta_t)^{-1/2} U'' (-\Delta_t)^{-1/2} \left( F_\pm + e^{\mp 2ik \beta t} F_\mp \right).
	\end{equation}
	
	We next collect the operator bounds needed to analyze \eqref{eq:scattering-system}. The bounds for \(R^2\) and \(\partial_t R\) will control the terms arising when we differentiate the normal-form correction for the oscillatory coupling.
	
	\begin{lemma}[Scattering operator bounds]
		\label{lem:scattering-operator-bounds}
		Under the assumptions on \(U\) in Theorem \ref{thm:stability}, for every \(f \in L^2(\mathbb R)\), the following bounds hold:
		\begin{align}
			\| Rf \|_{L^2} &\lesssim \left\| k (-\Delta_t)^{-1/2} f \right\|_{L^2}, \label{eq:scattering-R-bound} \\
			\| Kf \|_{L^2} &\lesssim \left\| k^2 (-\Delta_t)^{-1} f \right\|_{L^2}, \label{eq:scattering-K-bound} \\
			\left\| k (-\Delta_t)^{-1/2} U'' (-\Delta_t)^{-1/2} f \right\|_{L^2} &\lesssim \left\| k^2 (-\Delta_t)^{-1} f \right\|_{L^2}, \label{eq:scattering-curvature-bound} \\
			\| R^2 f \|_{L^2} &\lesssim \left\| k^2 (-\Delta_t)^{-1} f \right\|_{L^2}. \label{eq:scattering-R-square-bound} \\
			\| (\partial_t R) f \|_{L^2} &\lesssim \left\| k^2 (-\Delta_t)^{-1} f \right\|_{L^2}, \label{eq:scattering-R-derivative-bound}
		\end{align}
	\end{lemma}
	
	\begin{proof}
		We first prove \eqref{eq:scattering-R-bound}. By \eqref{eq:scattering-R} and \eqref{eq:laplacian-time-derivative}, we obtain
		\begin{equation}
			\label{eq:scattering-R-representation}
			R = -ik (-\Delta_t)^{-1/2} D_t U' (-\Delta_t)^{-1/2} + \frac{ik}{2} (-\Delta_t)^{-1/2} U'' (-\Delta_t)^{-1/2}.
		\end{equation}
		Unitary conjugation by \(\mathcal U(t)\) and Plancherel's theorem give
		\begin{equation}
			\label{eq:scattering-multiplier-bounds}
			\left\| (-\Delta_t)^{-1/2} D_t \right\|_{L^2 \to L^2} \le 1, \qquad \left\| (-\Delta_t)^{-1/2} \right\|_{L^2 \to L^2} \le 1.
		\end{equation}
		Using \(\| U' \|_{L^2 \to L^2} \le U'_{\sup}\) and \(\| U'' \|_{L^2 \to L^2} \le \epsilon\), we conclude that
		\[\| Rf \|_{L^2} \le \left( U'_{\sup} + \frac{\epsilon}{2} \right) \left\| k (-\Delta_t)^{-1/2} f \right\|_{L^2} \lesssim \left\| k (-\Delta_t)^{-1/2} f \right\|_{L^2}.\]
		
		To estimate \(K\), we use the cancellation between the two products in \eqref{eq:scattering-K}. Since
		\[\sigma \left( \mathcal U^{-1} \left( \partial_t (-\Delta_t)^{-1/2} \right) \mathcal U \right) (y,\eta) = \sigma \left( \mathcal U^{-1} ik \left[ U, (-\Delta_t)^{-1/2} \right] \mathcal U \right) (y,\eta) = i \sgn(k) \left[U(y), \left\langle \frac{\eta}{k} \right\rangle^{-1} \right],\]
		it follows from \eqref{eq:scattering-K} and the composition formula \eqref{eq:composition_formula} that
		\begin{align*}
			\sigma (\mathcal U^{-1} K \mathcal U) (y,\eta) &= - \frac{ik}{2} \left( \left\langle \frac{\eta}{k} \right\rangle^{-1} \# U(y) \# \left\langle \frac{\eta}{k} \right\rangle + \left\langle \frac{\eta}{k} \right\rangle \# U(y) \# \left\langle \frac{\eta}{k} \right\rangle^{-1} - 2U(y) \right) \\
			&= - \frac{ik}{4\pi} \int_{\mathbb R} \int_{\mathbb R} e^{-iz\lambda}U(y+z) \left( \left\langle \frac{\eta + \lambda}{k} \right\rangle^{-1} \left\langle \frac{\eta}{k} \right\rangle + \left\langle \frac{\eta + \lambda}{k} \right\rangle \left\langle\frac{\eta}{k} \right\rangle^{-1} - 2 \right) d\lambda dz.
		\end{align*}
		The expression in parentheses and its first \(\lambda\)-derivative vanish at \(\lambda=0\). Substituting the second-order Taylor expansion with integral remainder into the oscillatory integral and integrating twice by parts in \(z\), we obtain
		\[\sigma( \mathcal U^{-1} K \mathcal U) (y,\eta) = \frac{ik}{4\pi} \int_{\mathbb R} \int_{\mathbb R} \int_{0}^{1} e^{-iz\lambda} U''(y+z) (1-\theta) \left( \left\langle \frac{\eta}{k} \right\rangle \partial_\eta^2 \left\langle \frac{\eta + \theta\lambda}{k} \right\rangle^{-1} + \left\langle \frac{\eta}{k} \right\rangle^{-1} \partial_\eta^2 \left\langle \frac{\eta + \theta\lambda}{k} \right\rangle \right) d\theta d\lambda dz.\]
		Direct differentiation and Peetre's inequality give, for \(m \in \{0, 1, 2\}\) and \(l \in \{0, 1\}\),
		\[\left| \partial_\lambda^m \partial_\eta^l \left( \left\langle \frac{\eta}{k} \right\rangle \partial_\eta^2 \left\langle \frac{\eta + \theta\lambda}{k} \right\rangle^{-1} + \left\langle \frac{\eta}{k} \right\rangle^{-1} \partial_\eta^2 \left\langle \frac{\eta + \theta\lambda}{k} \right\rangle \right) \right| \lesssim|k|^{-2} \left\langle \frac{\eta}{k} \right\rangle^{-2} \langle \lambda \rangle^3.\]
		Integrating by parts in the oscillatory integral with \(1-\partial_\lambda^2\) and \((1-\partial_z^2)^3\), we obtain, for \(j, l \in \{0,1\}\),
		\[\left| \partial_y^j \partial_\eta^l \sigma (\mathcal U^{-1} K \mathcal U) (y,\eta) \right| \lesssim \epsilon |k|^{-1} \int_{\mathbb R} \int_{\mathbb R} \langle z \rangle^{-2} \langle \lambda \rangle^{-6} \left\langle \frac{\eta}{k} \right\rangle^{-2} \langle \lambda \rangle^3 d\lambda dz \lesssim \epsilon \left\langle \frac{\eta}{k} \right\rangle^{-2}.\]
		To prove \eqref{eq:scattering-K-bound}, it suffices to show that
		\[\| K k^{-2} (-\Delta_t) \|_{L^2 \to L^2} \lesssim 1.\]
		After conjugating by \(\mathcal U\), the symbol of this operator is
		\[\sigma \left( \mathcal U^{-1} K k^{-2} (-\Delta_t) \mathcal U \right) (y,\eta) = \sigma (\mathcal U^{-1} K \mathcal U) (y,\eta) \left\langle \frac{\eta}{k} \right\rangle^2.\]
		This symbol and its first derivatives in \(y\) and \(\eta\) are bounded in absolute value by \(C\epsilon\), as follows from the preceding estimate for \(\sigma (\mathcal U^{-1} K \mathcal U)\). The Calder\'on--Vaillancourt Theorem \ref{thm:CV} gives the required operator bound.
		
		Turning to \eqref{eq:scattering-curvature-bound}, we apply the argument of Lemma \ref{lem:shear-conjugation} to \((-\Delta_t)^{-1/2}\), with \(U'\) replaced by \(U''\). Since \(\| U'' \|_{W^{n,\infty}} \le \epsilon\), this gives
		\begin{equation}
			\label{eq:scattering-curvature-conjugation}
			\left\| (-\Delta_t)^{-1/2} U'' (-\Delta_t)^{1/2} \right\|_{L^2 \to L^2} \lesssim \epsilon.
		\end{equation}
		Consequently, using \(|k| \ge 1\),
		\[\left\| k (-\Delta_t)^{-1/2} U'' (-\Delta_t)^{-1/2} f \right\|_{L^2} \le \left\| (-\Delta_t)^{-1/2} U'' (-\Delta_t)^{1/2} \right\|_{L^2 \to L^2} \left\| k^2 (-\Delta_t)^{-1} f \right\|_{L^2} \lesssim \epsilon \left\| k^2(-\Delta_t)^{-1} f \right\|_{L^2}.\]
		This proves \eqref{eq:scattering-curvature-bound}.
		
		We now estimate \(R^2\). Since \(D_t\) commutes with \((-\Delta_t)^{-1/2}\), \eqref{eq:scattering-R-representation} can be rewritten as
		\[R = -ik D_t \left( (-\Delta_t)^{-1/2} U' (-\Delta_t)^{1/2} \right) (-\Delta_t)^{-1} + \frac{ik}{2} \left( (-\Delta_t)^{-1/2} U'' (-\Delta_t)^{1/2} \right) (-\Delta_t)^{-1}.\]
		Lemma \ref{lem:shear-conjugation} bounds the conjugated operator involving \(U'\), while \eqref{eq:scattering-curvature-conjugation} bounds the one involving \(U''\). Combining these bounds with \eqref{eq:scattering-multiplier-bounds}, we obtain
		\begin{equation}
			\label{eq:scattering-weighted-R-bound}
			\left\| k (-\Delta_t)^{-1/2} R f \right\|_{L^2} \lesssim \left\| k^2 (-\Delta_t)^{-1} f \right\|_{L^2}.
		\end{equation}
		Applying \eqref{eq:scattering-R-bound} to \(R f\), we conclude that
		\[\| R^2 f \|_{L^2} \lesssim \left\| k (-\Delta_t)^{-1/2} R f \right\|_{L^2} \lesssim \left\| k^2 (-\Delta_t)^{-1} f \right\|_{L^2},\]
		which proves \eqref{eq:scattering-R-square-bound}.
		
		To prove \eqref{eq:scattering-R-derivative-bound}, we need to compute \(\partial_tR\). Recalling the definitions of \(K\) and \(R\) in \eqref{eq:scattering-K} and \eqref{eq:scattering-R}, we have
		\[\left( \partial_t (-\Delta_t)^{-1/2} \right) (-\Delta_t)^{1/2} = R + K, \qquad (-\Delta_t)^{1/2} \left( \partial_t (-\Delta_t)^{-1/2} \right) = R - K.\]
		Differentiating \eqref{eq:scattering-R} and using \(\partial_t^2 \Delta_t = -2k^2 (U')^2\), we obtain
		\begin{align*}
			\partial_t R &= \frac{1}{2} \left( \partial_t (-\Delta_t)^{-1/2} \right) (\partial_t \Delta_t) (-\Delta_t)^{-1/2} + \frac{1}{2} (-\Delta_t)^{-1/2} (\partial_t \Delta_t) \left( \partial_t (-\Delta_t)^{-1/2} \right) + \frac{1}{2} (-\Delta_t)^{-1/2} (\partial_t^2 \Delta_t) (-\Delta_t)^{-1/2} \\
			&= (R + K) R + R (R - K) - k^2 (-\Delta_t)^{-1/2} (U')^2 (-\Delta_t)^{-1/2} \\
			&= 2R^2 + KR - RK -k^2 (-\Delta_t)^{-1/2} (U')^2 (-\Delta_t)^{-1/2}.
		\end{align*}
		Since \(\| k(-\Delta_t)^{-1/2} \|_{L^2 \to L^2} \le1\), \eqref{eq:scattering-R-bound}, \eqref{eq:scattering-K-bound}, and \eqref{eq:scattering-weighted-R-bound} give
		\[\| KRf \|_{L^2} \lesssim \left\| k^2 (-\Delta_t)^{-1} Rf \right\|_{L^2} \lesssim \left\| k (-\Delta_t)^{-1/2} Rf \right\|_{L^2} \lesssim \left\| k^2 (-\Delta_t)^{-1} f \right\|_{L^2},\]
		\[\| RKf \|_{L^2} \lesssim \left\| k (-\Delta_t)^{-1/2} Kf \right\|_{L^2} \lesssim \| Kf \|_{L^2} \lesssim \left\| k^2 (-\Delta_t)^{-1} f \right\|_{L^2}.\]
		Moreover, the operator identity
		\[k^2 (-\Delta_t)^{-1/2} (U')^2 (-\Delta_t)^{-1/2} = \left( (-\Delta_t)^{-1/2} U' (-\Delta_t)^{1/2} \right)^2 k^2 (-\Delta_t)^{-1},\]
		and Lemma \ref{lem:shear-conjugation} yield
		\[\left\| k^2 (-\Delta_t)^{-1/2} (U')^2 (-\Delta_t)^{-1/2} f \right\|_{L^2} \lesssim \left\| k^2 (-\Delta_t)^{-1} f \right\|_{L^2}.\]
		Combining these bounds with \eqref{eq:scattering-R-square-bound} proves \eqref{eq:scattering-R-derivative-bound}.
	\end{proof}
	
	\subsection{Proof of Theorem \ref{thm:scattering}}
	
	We begin with strong scattering of the phase-corrected amplitudes. Their limits will then be used to identify the leading asymptotics of the physical variables.
	
	\begin{proof}
		\noindent\textbf{Step 1. Strong scattering.} Since we will sum over the nonzero horizontal modes, we restore the subscript \(k\) on the transformed variables and operators in this step.
		
		To cancel the oscillatory cross-channel term in \eqref{eq:scattering-system}, we introduce the normal-form variables
		\begin{equation}
			\label{eq:scattering-normal-form}
			Z_{\pm,k} = F_{\pm,k} \pm \frac{e^{\mp 2ik \beta t}}{2ik \beta} R_k F_{\mp,k}.
		\end{equation}
		We need to show that the correction tends to zero, so that \(Z_{\pm,k}\) and \(F_{\pm,k}\) have the same scattering limit. By \eqref{eq:scattering-R-bound}, it suffices to estimate \(\| k (-\Delta_{t,k})^{-1/2} F_{\pm,k} \|_{L_y^2}\). The Cauchy--Schwarz inequality gives
		\[\| k (-\Delta_{t,k})^{-1/2} F_{\pm,k} \|_{L_y^2}^2 \le \| F_{\pm,k} \|_{L_y^2} \| k^2 (-\Delta_{t,k})^{-1} F_{\pm,k} \|_{L_y^2}.\]
		By the potential representation \eqref{eq:potential-representation},
		\[\| F_{\pm,k} \|_{L_y^2} = \| (V_k \pm B_k) \|_{L_y^2}, \qquad \| k^2 (-\Delta_{t,k})^{-1} F_{\pm,k} \|_{L_y^2} = \| k^2 (-\Delta_{t,k})^{-1} (V_k\pm B_k) \|_{L_y^2}.\]
		Combining the identities above with \eqref{eq:lower-endpoint} and \eqref{eq:mode-uniform-bound}, and then applying the Cauchy--Schwarz estimate, we obtain
		\begin{equation}
			\label{eq:scattering-decay-inputs}
			\| k (-\Delta_{t,k})^{-1/2} F_{\pm,k}  (t)\|_{L_y^2} \lesssim \| (v_{\mathrm{in},k},b_{\mathrm{in},k}) \|_{H_y^2} \langle t \rangle^{-1}, \qquad \|k^2 (-\Delta_{t,k})^{-1} F_{\pm,k} (t) \|_{L_y^2} \lesssim \| (v_{\mathrm{in},k},b_{\mathrm{in},k}) \|_{H_y^2} \langle t \rangle^{-2}.
		\end{equation}
		Hence
		\begin{equation}
			\label{eq:scattering-normal-form-error}
			\| Z_{\pm,k} (t) - F_{\pm,k} (t) \|_{L_y^2} \le \frac{1}{2|k\beta|} \| R_k F_{\mp,k} (t) \|_{L_y^2} \lesssim \| (v_{\mathrm{in},k},b_{\mathrm{in},k}) \|_{H_y^2} \langle t \rangle^{-1}.
		\end{equation}
		
		Differentiating \(Z_{\pm,k}\) and substituting \eqref{eq:scattering-system}, we see that the oscillatory terms involving \(R_k F_{\mp,k}\) cancel. Hence
		\[\partial_t Z_{\pm,k} = K_k F_{\pm,k} - \frac{ik}{2} (-\Delta_{t,k})^{-1/2} U'' (-\Delta_{t,k})^{-1/2} \left( F_{\pm,k} + e^{\mp 2ik \beta t} F_{\mp,k} \right) \pm \frac{e^{\mp 2ik \beta t}}{2ik \beta} ((\partial_tR_k) F_{\mp,k} + R_k \partial_t F_{\mp,k}).\]
		By \eqref{eq:scattering-K-bound}, \eqref{eq:scattering-curvature-bound}, \eqref{eq:scattering-R-derivative-bound}, and \eqref{eq:scattering-decay-inputs}, we have
		\[\| K_k F_{\pm,k} (t) \|_{L_y^2} \lesssim \| (v_{\mathrm{in},k},b_{\mathrm{in},k}) \|_{H_y^2} \langle t \rangle^{-2}.\]
		\[\left\| \frac{ik}{2} (-\Delta_{t,k})^{-1/2} U'' (-\Delta_{t,k})^{-1/2} \left( F_{\pm,k} (t) + e^{\mp 2ik \beta t} F_{\mp,k} (t) \right) \right\|_{L_y^2} \lesssim \| (v_{\mathrm{in},k},b_{\mathrm{in},k}) \|_{H_y^2} \langle t \rangle^{-2}.\]
		\[\left\| \frac{e^{\mp 2ik \beta t}}{2ik \beta} (\partial_tR_k) F_{\mp,k} (t) \right\|_{L_y^2} \lesssim \| (v_{\mathrm{in},k},b_{\mathrm{in},k}) \|_{H_y^2} \langle t \rangle^{-2}.\]
		To estimate the remaining term, we use \eqref{eq:scattering-system} once more:
		\[R_k \partial_t F_{\mp,k} = R_k K_kF_{\mp,k} + e^{\pm 2ik \beta t} R_k^2 F_{\pm,k} - \frac{ik}{2} R_k (-\Delta_{t,k})^{-1/2} U'' (-\Delta_{t,k})^{-1/2} \left( F_{\mp,k} + e^{\pm 2ik \beta t} F_{\pm,k} \right).\]
		\eqref{eq:scattering-R-bound} implies that the operator \(R_k\) is uniformly bounded on \(L_y^2\). Thus \eqref{eq:scattering-K-bound}, \eqref{eq:scattering-curvature-bound}, \eqref{eq:scattering-R-square-bound}, and \eqref{eq:scattering-decay-inputs} also give
		\[\| R_k \partial_t F_{\mp,k} (t) \|_{L_y^2} \lesssim \| (v_{\mathrm{in},k},b_{\mathrm{in},k}) \|_{H_y^2} \langle t \rangle^{-2}.\]
		Hence, we conclude that
		\begin{equation}
			\label{eq:scattering-normal-form-derivative}
			\| \partial_t Z_{\pm,k} (t) \|_{L_y^2} \lesssim \| (v_{\mathrm{in},k},b_{\mathrm{in},k}) \|_{H_y^2} \langle t \rangle^{-2}.
		\end{equation}
		
		For \(0 \le t_1 \le t_2\), integrating \eqref{eq:scattering-normal-form-derivative} gives
		\[\| Z_{\pm,k} (t_2) - Z_{\pm,k} (t_1) \|_{L_y^2} \lesssim \| (v_{\mathrm{in},k},b_{\mathrm{in},k}) \|_{H_y^2} \int_{t_1}^{t_2} \langle \tau \rangle^{-2} d\tau \lesssim \| (v_{\mathrm{in},k},b_{\mathrm{in},k}) \|_{H_y^2} \langle t_1 \rangle^{-1}.\]
		Thus \(Z_{\pm,k}(t)\) converges strongly in \(L_y^2\) to a unique profile \(F_{\pm,\infty,k}\). By \eqref{eq:scattering-normal-form-error}, \(F_{\pm,k}(t)\) has the same limit. Moreover, by Minkowski’s integral inequality
		\begin{equation}
			\label{eq:modewise-strong-scattering}
			\| F_{\pm,k} (t) - F_{\pm,\infty,k} \|_{L_y^2} \le \| F_{\pm,k} (t) - Z_{\pm,k}(t) \|_{L_y^2} + \int_{t}^{+\infty} \| \partial_\tau Z_{\pm,k} (\tau) \|_{L_y^2} d\tau \lesssim \| (v_{\mathrm{in},k},b_{\mathrm{in},k}) \|_{H_y^2} \langle t \rangle^{-1}.
		\end{equation}
		At \(t = 0\), this estimate and \(\| F_{\pm,k} (0) \|_{L_y^2} = \| (v_{\mathrm{in},k} \pm b_{\mathrm{in},k}) \|_{L_y^2}\) show that \(\| F_{\pm,\infty,k} \|_{L_y^2} \lesssim \| (v_{\mathrm{in},k},b_{\mathrm{in},k}) \|_{H_y^2}\). Since \((v_{\mathrm{in}},b_{\mathrm{in}})\in H_{x,y}^2\subset L_x^2H_y^2\), the sequence \(\{ \| F_{\pm,\infty,k} \|_{L_y^2} \}_{k \neq 0}\) is square summable. These Fourier coefficients determine profiles \(F_{\pm,\infty} \in L^2_{\neq}\). Moreover, by Plancherel's theorem and the modewise scattering estimate, we obtain
		\[\| F_\pm (t) - F_{\pm,\infty} \|_{L_x^2L_y^2} \lesssim \| (v_{\mathrm{in}},b_{\mathrm{in}}) \|_{H_{x,y}^2} \langle t \rangle^{-1}.\]
		
		\noindent\textbf{Step 2. Horizontal-component asymptotics.} Fixing \(k \neq 0\), we again suppress the mode index on the transformed variables, \(F_\pm\), and the associated operators, while retaining it on the physical Fourier coefficients, initial data, and limiting profiles.
		
		We begin with the horizontal components. By \eqref{eq:potential-representation},
		\[V_1 \pm B_1 = -D_t \Delta_t^{-1} (\Omega\pm J).\]
		 On the other hand, the definition of \(H_s\) in \eqref{eq:lower-weight} gives
		\begin{equation}
			\label{eq:terminal-lower-weight}
			H_t (t) = \mathcal U(t) \op \left( h \left( \frac{\eta}{k} \right) \right) \mathcal U(t)^{-1} = 1 -  i\sgn(k) D_t (-\Delta_t)^{-1/2},
		\end{equation}
		where we used \(h(\xi) = 1 + \xi \langle \xi \rangle^{-1}\). Thus \eqref{eq:sheared-amplitudes} implies
		\[e^{\mp ik\beta t}(V_1\pm B_1)+i\sgn(k)F_\pm=i\sgn(k)H_tF_\pm.\]
		To bound the right-hand side, we use the potential representation \eqref{eq:potential-representation} again and the lower-order energy estimate in Proposition \ref{prop:lower-energy},
		\begin{equation}
			\label{eq:lower-amplitude-decay}
			\| H_t F_\pm (t) \|_{L^2} = \| H_t (V \pm B) (t) \|_{L^2} \lesssim \| (v_{\mathrm{in},k},b_{\mathrm{in},k}) \|_{H^2} \langle t \rangle^{-2}.
		\end{equation}
		Combining this bound with \eqref{eq:modewise-strong-scattering} and using \eqref{eq:shear-transform}, we obtain
		\[\left\| (v_1 \pm b_1)_k (t) + i\sgn(k) e^{-ik(U\mp\beta)t} F_{\pm,\infty,k} \right\|_{L^2} \lesssim \| (v_{\mathrm{in},k},b_{\mathrm{in},k}) \|_{H^2} \langle t \rangle^{-1}.\]
		
		\noindent\textbf{Step 3. Vorticity--current asymptotics.} We next turn to the vorticity--current pair. The estimate for \(\partial_yF_\pm\) proved in this step will also be used for the vertical components. We likewise express \(\Omega\pm J\) in terms of \(F_\pm\). By \eqref{eq:sheared-amplitudes},
		\[\Omega \pm J = e^{\pm ik \beta t} (-\Delta_t)^{1/2} F_\pm.\]
		To identify the leading behavior of \(\Omega\pm J\), we therefore need to analyze \((-\Delta_t)^{1/2} F_\pm\). The square-root operator \((-\Delta_t)^{1/2} = (k^2-D_t^2)^{1/2}\) is less explicit than \(D_t = \partial_y - ikU't\), whose time dependence can be analyzed directly. Since both are first-order operators, we seek to compare \((-\Delta_t)^{1/2}\) with \(i\sgn(k)D_t\) on \(F_\pm\). This would allow us to obtain the leading term by analyzing \(D_t\). We will show that their difference, when applied to \(F_\pm\), remains bounded in \(L^2\) as \(t\to\infty\).
		
		Indeed, the identity \(\langle \xi \rangle + \xi = \langle \xi \rangle h(\xi)\), together with the definitions of \(G_s\) and \(H_s\) in \eqref{eq:higher-weight} and \eqref{eq:lower-weight}, gives
		\[(-\Delta_t)^{1/2} - i\sgn(k) D_t = |k| G_t(t)^{1/2} H_t(t).\]
		Since \(h \le 2\), we have \(\| H_t \|_{L^2 \to L^2} \le 2\). Using the commutativity of \(G_t\) and \(H_t\) and the Cauchy--Schwarz inequality, we obtain
		\[\left\| ((-\Delta_t)^{1/2} - i\sgn(k) D_t) F_\pm (t) \right\|_{L^2}^2 = k^2 \langle G_t F_\pm (t), H_t^2 F_\pm (t) \rangle \le 2k^2 \| G_t F_\pm (t) \|_{L^2} \| H_t F_\pm (t) \|_{L^2}.\]
		By Proposition \ref{prop:higher-energy}, we have
		\[\| G_t F_\pm (t) \|_{L^2}  = \| G_t (V \pm B) (t) \|_{L^2} \lesssim \| (v_{\mathrm{in},k},b_{\mathrm{in},k}) \|_{H^2} \langle t \rangle^2.\]
		Combining this with \eqref{eq:lower-amplitude-decay} in the preceding inequality, we obtain
		\[\left\| ((-\Delta_t)^{1/2} - i\sgn(k) D_t) F_\pm (t) \right\|_{L^2} \lesssim|k| \| (v_{\mathrm{in},k},b_{\mathrm{in},k}) \|_{H^2}.\]
		Thus it remains to prove
		\[i\sgn(k) D_t F_\pm (t) = |k| U' t F_{\pm,\infty,k} + o_{L^2}(t).\]
		
		We now analyze \(D_tF_\pm\). Since \(D_t = \partial_y - ikU't\),
		\[i\sgn(k) D_t F_\pm = |k| U't F_\pm + i\sgn(k) \partial_y F_\pm.\]
		By \eqref{eq:modewise-strong-scattering} and the boundedness of \(U'\), it suffices to show that \(\partial_y F_\pm (t) = o_{L^2}(t)\). By \eqref{eq:mode-uniform-bound}, \eqref{eq:higher-endpoint}, and the Cauchy--Schwarz inequality, we have
		\[\| D_t F_\pm (t) \|_{L^2}^2 \le \langle F_\pm(t), (-\Delta_t) F_\pm(t) \rangle \le \| F_\pm(t) \|_{L^2} \| (-\Delta_t) F_\pm(t) \|_{L^2} \lesssim k^2 \| (v_{\mathrm{in},k},b_{\mathrm{in},k}) \|_{H^2}^2 \langle t \rangle^2.\]
		Combining this with the identity
		\[\partial_y^2 = (D_t + ikU't)^2 = D_t^2 + 2ikU't D_t + ikU''t - k^2 (U')^2 t^2,\]
		we have
		\[\| \partial_y^2 F_\pm (t) \|_{L^2} \lesssim k^2 \| (v_{\mathrm{in},k},b_{\mathrm{in},k}) \|_{H^2} \langle t \rangle^2.\]
		
		To show that \(t^{-1} \partial_y F_\pm(t) \to 0\) in \(L^2\), we approximate the scattering profile by smooth compactly supported functions and combine strong convergence with the preceding second-derivative bound.  For any \(f \in C_c^\infty (\mathbb R)\), the \(L^2\) interpolation inequality applied to \(F_\pm(t) - f \in H^2\) gives
		\[t^{-2} \| \partial_y (F_\pm(t) - f) \|_{L^2}^2 \le \| F_\pm (t) - f \|_{L^2} \times t^{-2} \| \partial_y^2 (F_\pm (t) - f) \|_{L^2}.\]
		By \eqref{eq:modewise-strong-scattering}, the first factor on the right converges to \(\| F_{\pm,\infty,k} - f \|_{L^2}\). For the second factor, the preceding derivative bound for \(\partial_y^2 F_\pm\) yields
		\[\limsup_{t \to \infty} t^{-2} \| \partial_y^2 (F_\pm(t) - f) \|_{L^2} \le \limsup_{t \to \infty} \left( t^{-2} \| \partial_y^2 F_\pm (t) \|_{L^2} + t^{-2} \| \partial_y^2 f \|_{L^2} \right) \lesssim k^2 \| (v_{\mathrm{in},k},b_{\mathrm{in},k}) \|_{H^2},\]
		with a bound independent of \(f\). Therefore,
		\[\limsup_{t \to \infty} t^{-2} \| \partial_y (F_\pm(t) - f) \|_{L^2}^2\lesssim \| F_{\pm,\infty,k} - f \|_{L^2} k^2 \| (v_{\mathrm{in},k},b_{\mathrm{in},k}) \|_{H^2}.\]
		Thus the triangle inequality implies
		\[\limsup_{t\to\infty} t^{-1} \| \partial_y F_\pm (t) \|_{L^2}\lesssim \| F_{\pm,\infty,k} - f \|_{L^2}^{1/2} |k| \| (v_{\mathrm{in},k},b_{\mathrm{in},k}) \|_{H^2}^{1/2}.\]
		This holds for every fixed \(f \in C_c^\infty\). Since these functions are dense in \(L^2\), we conclude that
		\begin{equation}
			\label{eq:amplitude-derivative-decay}
			t^{-1} \| \partial_y F_\pm (t) \|_{L^2} \longrightarrow 0.
		\end{equation}
		
		Combining the preceding estimates, we obtain the leading asymptotics
		\[(\omega\pm j)_k (t,y) = |k| U'(y) t e^{-ik (U(y) \mp \beta) t} F_{\pm,\infty,k}(y) + o_{L_y^2} (t).\]
		
		\noindent\textbf{Step 4. Vertical-component asymptotics.} Finally, we consider the vertical components of \(v \pm b\), also expressing them in terms of \(F_\pm\). Specifically, \eqref{eq:potential-representation} and \eqref{eq:sheared-amplitudes} give
		\[e^{ik(U(y) \mp \beta) t} (v_2 \pm b_2)_k (t) = -ik (-\Delta_t)^{-1/2} F_\pm(t).\]
		Since \(U'\) is bounded above and below by positive constants, the asserted vertical-component asymptotics reduce to proving that \(|k| U' t (-\Delta_t)^{-1/2} F_\pm (t)\) converges strongly in \(L^2\) to \(F_{\pm,\infty,k}\). Combining \eqref{eq:terminal-lower-weight} with \(D_t = \partial_y - ikU't\), we obtain
		\[|k|U't (-\Delta_t)^{-1/2} F_\pm = F_\pm - H_t F_\pm - i\sgn(k) \partial_y \left( (-\Delta_t)^{-1/2} F_\pm \right).\]
		The first term on the right converges to \(F_{\pm,\infty,k}\) by \eqref{eq:modewise-strong-scattering}, and the second tends to zero by \eqref{eq:lower-amplitude-decay}. Hence it remains to show that the third term also tends to zero in \(L^2\).
		
		To estimate the third term, we move \(\partial_y\) onto \(F_\pm\) so that \eqref{eq:amplitude-derivative-decay} can be used. Since \(\partial_y = D_t + ikU't\) and \(D_t\) commutes with \((-\Delta_t)^{-1/2}\), we have
		\[\partial_y \left( (-\Delta_t)^{-1/2} F_\pm \right) = (-\Delta_t)^{-1/2} \partial_y F_\pm + ikt [U',(-\Delta_t)^{-1/2}] F_\pm.\]
		We estimate the first term by the Cauchy--Schwarz inequality:
		\[\| (-\Delta_t)^{-1/2} \partial_y F_\pm \|_{L^2}^2 \le \| \partial_y F_\pm \|_{L^2} \| (-\Delta_t)^{-1} \partial_y F_\pm \|_{L^2}.\]
		In view of \eqref{eq:amplitude-derivative-decay}, it suffices here to establish \(t^{-1}\) decay of the second factor. To apply the available inverse-Laplacian bounds, we use the same decomposition of \(\partial_y\) to write
		\[(-\Delta_t)^{-1} \partial_y F_\pm = (-\Delta_t)^{-1} D_t F_\pm + ikt (-\Delta_t)^{-1} U' F_\pm.\]
		\eqref{eq:scattering-multiplier-bounds} and \eqref{eq:scattering-decay-inputs} yield
		\[\| (-\Delta_t)^{-1} D_t F_\pm (t) \|_{L^2} \le \| (-\Delta_t)^{-1/2} F_\pm (t) \|_{L^2} \lesssim \| (v_{\mathrm{in},k},b_{\mathrm{in},k}) \|_{H^2} \langle t \rangle^{-1}.\]
		Applying Lemma \ref{lem:shear-conjugation} to \((-\Delta_t)^{-1}\) and using \eqref{eq:scattering-decay-inputs} again, we also obtain
		\[\| ikt (-\Delta_t)^{-1} U' F_\pm (t) \|_{L^2} \lesssim (1 + \epsilon) \| kt U' (-\Delta_t)^{-1} F_\pm (t) \|_{L^2} \lesssim \| (v_{\mathrm{in},k},b_{\mathrm{in},k}) \|_{H^2} \langle t \rangle^{-1}.\]
		Together with \eqref{eq:amplitude-derivative-decay}, these bounds imply
		\[\| (-\Delta_t)^{-1/2} \partial_y F_\pm (t) \|_{L^2} \longrightarrow 0.\]
		
		We now estimate the commutator term. The composition formula \eqref{eq:composition_formula}, a first-order Taylor expansion, and integration by parts in \(z\) give
		\[\sigma \left( \mathcal U^{-1} [U',(-\Delta_t)^{-1/2}] (-\Delta_t) \mathcal U \right) (y,\eta) = \frac{i}{2\pi} \int_{\mathbb R} \int_{\mathbb R} \int_{0}^{1} e^{-iz\lambda} U''(y+z) \left( \partial_\eta \left( k^2 + (\eta + \theta\lambda)^2 \right)^{-1/2} \right) (k^2 + \eta^2) d\theta d\lambda dz.\]
		Direct differentiation and Peetre's inequality show that, for \(l \in \{0, 1\}\) and \(m \in \{0, 2\}\),
		\[\left| \partial_\eta^l \partial_\lambda^m \left[ \left( \partial_\eta \left( k^2 + (\eta + \theta \lambda)^2 \right)^{-1/2} \right) (k^2 + \eta^2) \right] \right| \lesssim \langle \lambda \rangle^2.\]
		Integrating by parts in the oscillatory variables as in the proof of Lemma \ref{lem:shear-conjugation}, we obtain
		\[\left| \partial_y^j \partial_\eta^l \sigma \left( \mathcal U^{-1} [U',(-\Delta_t)^{-1/2}] (-\Delta_t) \mathcal U \right) (y,\eta) \right| \lesssim \epsilon, \qquad j, l \in \{0,1\}.\]
		The Calder\'on--Vaillancourt Theorem \ref{thm:CV} therefore gives
		\[\left\| [U',(-\Delta_t)^{-1/2}] (-\Delta_t) \right\|_{L^2 \to L^2} \lesssim \epsilon.\]
		By \eqref{eq:scattering-decay-inputs}, it follows that
		\[\| ikt [U',(-\Delta_t)^{-1/2}] F_\pm (t) \|_{L^2} \lesssim \epsilon \| kt (-\Delta_t)^{-1} F_\pm (t) \|_{L^2} \lesssim \epsilon \| (v_{\mathrm{in},k},b_{\mathrm{in},k}) \|_{H^2} \langle t \rangle^{-1} \longrightarrow 0.\]
		
		Combining the preceding estimates yields the asserted vertical-component asymptotics, completing the proof of Theorem \ref{thm:scattering}.
	\end{proof}
	
	\section{Asymptotic completeness}
	\label{sec5}
	
	We now prove Theorem \ref{thm:completeness}. We first extend the scattering map to the energy space \(\mathcal H\) using the uniform energy bounds, and then establish its bijectivity and the two-sided bounds. When working with Fourier coefficients, we retain the mode index \(k\) throughout this section.
	
	\begin{proof}[Proof of Theorem \ref{thm:completeness}]
		\noindent\textbf{Step 1. Extension and two-sided bounds.} By the definition of \(F_\pm\) in \eqref{eq:phase-corrected-amplitudes} and the parallelogram identity, we have
		\begin{equation}
			\label{eq:amplitude-parallelogram}
			\| F_+(t) \|_{L_x^2L_y^2}^2 + \| F_-(t) \|_{L_x^2L_y^2}^2 = 2 \left( \| v(t) \|_{L_x^2L_y^2}^2 + \| b(t) \|_{L_x^2L_y^2}^2 \right).
		\end{equation}
		The same identity holds for each nonzero horizontal mode. In particular, the amplitudes are well defined for solutions with initial data in \(\mathcal H\). Using \eqref{eq:mode-uniform-bound} and Plancherel's theorem,
		\begin{equation}
			\label{eq:amplitude-uniform-bound}
			\| (F_+,F_-) (t) \|_{L_x^2L_y^2} \lesssim \| (v_{\mathrm{in}},b_{\mathrm{in}}) \|_{L_x^2L_y^2}.
		\end{equation}
		
		Fix \((v_{\mathrm{in}},b_{\mathrm{in}}) \in \mathcal H\), and choose \((v_{\mathrm{in}}^{(m)},b_{\mathrm{in}}^{(m)}) \in \mathcal H \cap H_{x,y}^2\) converging to it in \(L_x^2L_y^2\). Let \(F_\pm^{(m)}\) be the amplitudes corresponding to these regular data. By linearity and \eqref{eq:amplitude-uniform-bound},
		\[\sup_{t \ge 0} \| (F_+,F_-)(t) - (F_+^{(m)},F_-^{(m)})(t) \|_{L_x^2L_y^2} \lesssim \| (v_{\mathrm{in}},b_{\mathrm{in}}) - (v_{\mathrm{in}}^{(m)},b_{\mathrm{in}}^{(m)}) \|_{L_x^2L_y^2} \longrightarrow 0, \qquad m \to \infty.\]
		For each \(m\), Theorem \ref{thm:scattering} gives strong convergence of \(F_\pm^{(m)}(t)\) as \(t \to \infty\). Choosing \(m\) large in the uniform approximation and then taking \(t,s\) large for this fixed \(m\), we see that \(F_\pm(t)\) form a Cauchy family in \(L_x^2L_y^2\) as \(t \to \infty\). Their limits \(F_{\pm,\infty}\) therefore exist. Thus the scattering map \(\mathcal S\) is well defined on all of \(\mathcal H\). It is linear, and passing to the limit in \eqref{eq:amplitude-uniform-bound} gives the asserted upper estimate.
		
		For the lower bound, the time-reversal argument leading to \eqref{eq:uniform-reverse-energy-bound} uses only the basic \(L^2\) energy estimate and therefore applies to every initial datum in \(\mathcal H\). The parallelogram identity \eqref{eq:amplitude-parallelogram} and the strong convergence of \(F_\pm(t)\) then give, upon letting \(t\to\infty\),
		\[\| (v_{\mathrm{in}},b_{\mathrm{in}}) \|_{L_x^2L_y^2} \lesssim \| (F_{+,\infty},F_{-,\infty}) \|_{L_x^2L_y^2}.\]
		Thus \(\mathcal S\) is injective, its inverse on \(\ran \mathcal S\) is bounded, and its range is closed.
		
		\noindent\textbf{Step 2. Surjectivity.} Since \(\ran \mathcal S\) is closed, it suffices to show that it contains every pair \((f_+,f_-)\) with finitely many nonzero horizontal modes and Fourier coefficients in \(C_c^\infty (\mathbb R)\). Fix such a pair. For each \(t \ge 0\), prescribe \(F_\pm^{(t)}(t)=f_\pm\). By \eqref{eq:phase-corrected-amplitudes} and the Biot--Savart formula, this condition determines a unique state at time \(t\), and backward evolution of \eqref{eq:linear} then determines unique initial data \((v_{\mathrm{in}}^{(t)},b_{\mathrm{in}}^{(t)}) \in \mathcal H\). We will show that these initial data converge in \(\mathcal H\) as \(t\to\infty\) and that their limit has scattering profiles \((f_+,f_-)\).
		
		To show that the amplitudes remain close to \(f_\pm\) after time \(t\), we need the inverse-Laplacian decay used in the normal-form argument. We obtain it from the lower-order weight with terminal time \(s \ge t \ge 0\). For each fixed \(k \neq 0\) and \(f_{\pm,k} \in C_c^\infty\), we claim that
		\begin{equation}
			\label{eq:terminal-shift-bound}
			\| H_s(t) f_{\pm,k} \|_{L_y^2} \lesssim \| f_{\pm,k} \|_{H_y^2} \langle s \rangle^{-2},
		\end{equation}
		where the \(H_y^2\) norms on the right are finite, and only finitely many modes are nonzero by our choice of \(f_\pm\). Indeed, repeating the operator estimate for \(\partial_t H_s\) in the proof of Proposition \ref{prop:lower-energy} without the factor \(S\), we obtain
		\[\frac{d}{dt} \frac{1}{2} \| H_s(t) f_{\pm,k} \|_{L_y^2}^2 = \re \langle H_s(t) f_{\pm,k},(\partial_t H_s (t)) f_{\pm,k} \rangle \le \left( -\frac{U'_{\inf}}{2} + C \epsilon \right) \| Q_s (t) H_s (t) f_{\pm,k} \|_{L_y^2}^2 \le 0.\]
		Thus
		\[\| H_s(t) f_{\pm,k} \|_{L_y^2} \le \| H_s(0) f_{\pm,k} \|_{L_y^2} \lesssim \| k^{-2} (-\Delta_k) f_{\pm,k} \|_{L_y^2} \langle s \rangle^{-2} \lesssim \| f_{\pm,k} \|_{H_y^2} \langle s \rangle^{-2},\]
		which proves \eqref{eq:terminal-shift-bound}.
		
		Applying Proposition \ref{prop:lower-energy} to the solution with \(F_\pm^{(t)}(t) = f_\pm\) on \([t,s]\), together with \eqref{eq:amplitude-parallelogram} and the terminal multiplier bound for \(H_s(s)\) and \eqref{eq:terminal-shift-bound}, we have
		\begin{align*}
			\| k^2 (-\Delta_{s,k})^{-1} (F_{+,k}^{(t)},F_{-,k}^{(t)}) (s) \|_{L_y^2} &\lesssim \| H_s(s) (F_{+,k}^{(t)},F_{-,k}^{(t)}) (s) \|_{L_y^2} \lesssim \| H_s(t) (F_{+,k}^{(t)},F_{-,k}^{(t)}) (t) \|_{L_y^2} \\
			&= \| H_s(t) (f_{+,k},f_{-,k}) \|_{L_y^2} \lesssim \| (f_{+,k},f_{-,k}) \|_{H_y^2} \langle s \rangle^{-2}.
		\end{align*}
		The basic energy estimate in Proposition \ref{prop:basic-energy}, applied on \([t,s]\), and \eqref{eq:amplitude-parallelogram} give
		\[\| (F_{+,k}^{(t)},F_{-,k}^{(t)}) (s) \|_{L_y^2} \lesssim \|(f_{+,k},f_{-,k}) \|_{L_y^2}.\]
		Combining this with the preceding estimate and the Cauchy--Schwarz inequality, we obtain
		\begin{equation}
			\label{eq:terminal-half-order-bound}
			\| k (-\Delta_{s,k})^{-1/2} (F_{+,k}^{(t)},F_{-,k}^{(t)}) (s) \|_{L_y^2} \lesssim \| (f_{+,k},f_{-,k}) \|_{H_y^2} \langle s \rangle^{-1}.
		\end{equation}
		
		Let \(Z_{\pm,k}^{(t)}\) be the normal-form variables in \eqref{eq:scattering-normal-form} associated with \(F_{\pm,k}^{(t)}\). The estimates for the normal-form correction and its time derivative in the proof of Theorem \ref{thm:scattering} apply with the preceding bounds in place of \eqref{eq:scattering-decay-inputs}, giving
		\[\| Z_{\pm,k}^{(t)}(s) - F_{\pm,k}^{(t)}(s) \|_{L_y^2} \lesssim \| (f_{+,k},f_{-,k}) \|_{H_y^2} \langle s \rangle^{-1}, \qquad \| \partial_s Z_{\pm,k}^{(t)}(s) \|_{L_y^2} \lesssim \| (f_{+,k},f_{-,k}) \|_{H_y^2} \langle s \rangle^{-2}.\]
		Since \(F_{\pm,k}^{(t)}(t) = f_{\pm,k}\), integration over \([t,s]\) implies
		\begin{equation}
			\label{eq:terminal-forward-tail}
			\| F_{\pm,k}^{(t)}(s) - f_{\pm,k} \|_{L_y^2} \le \| F_{\pm,k}^{(t)}(s) - Z_{\pm,k}^{(t)}(s) \|_{L_y^2} + \int_{t}^{s} \| \partial_\tau Z_{\pm,k}^{(t)}(\tau) \|_{L_y^2} d\tau + \| Z_{\pm,k}^{(t)}(t) - f_{\pm,k} \|_{L_y^2} \lesssim \| (f_{+,k},f_{-,k}) \|_{H_y^2} \langle t \rangle^{-1}.
		\end{equation}
		Summing \eqref{eq:terminal-forward-tail} over the finitely many nonzero modes and using Plancherel's theorem, we obtain
		\[\| (F_+^{(t)},F_-^{(t)}) (s) - (f_+,f_-) \|_{L_x^2L_y^2} \lesssim \| (f_+,f_-) \|_{L_x^2H_y^2} \langle t \rangle^{-1}.\]
		The solution constructed with terminal time \(s\) satisfies \(F_\pm^{(s)} (s) = f_\pm\). Hence the preceding estimate controls the difference between the amplitudes of the two solutions at time \(s\). Their difference also solves \eqref{eq:linear}, so \eqref{eq:uniform-reverse-energy-bound} and \eqref{eq:amplitude-parallelogram} give
		\begin{align*}
			\| (v_{\mathrm{in}}^{(t)},b_{\mathrm{in}}^{(t)}) - (v_{\mathrm{in}}^{(s)},b_{\mathrm{in}}^{(s)}) \|_{L_x^2L_y^2} &\lesssim \| (v^{(t)},b^{(t)}) (s) - (v^{(s)},b^{(s)}) (s) \|_{L_x^2L_y^2} \\
			&= 2^{-1/2} \| (F_+^{(t)},F_-^{(t)}) (s) - (f_+,f_-) \|_{L_x^2L_y^2} \\
			&\lesssim \| (f_+,f_-) \|_{L_x^2H_y^2} \langle t \rangle^{-1}.
		\end{align*}
		Hence the initial data \((v_{\mathrm{in}}^{(t)},b_{\mathrm{in}}^{(t)})\) form a Cauchy family in \(\mathcal H\) as \(t\to\infty\).
		
		Let \((v_{\mathrm{in}},b_{\mathrm{in}})\in\mathcal H\) be the limit of these initial data, and let \(F_\pm\) be its amplitudes. Since \(F_\pm^{(t)}(t) = f_\pm\), the uniform bound \eqref{eq:amplitude-uniform-bound} applied to the difference of the two solutions gives
		\[\|(F_+,F_-)(t)-(f_+,f_-)\|_{L_x^2L_y^2}\lesssim\|(v_{\mathrm{in}},b_{\mathrm{in}})-(v_{\mathrm{in}}^{(t)},b_{\mathrm{in}}^{(t)})\|_{L_x^2L_y^2}\longrightarrow0.\]
		Thus \(\mathcal S (v_{\mathrm{in}},b_{\mathrm{in}}) = (f_+,f_-)\). Such pairs are dense in the target space, and \(\ran\mathcal S\) is closed by Step 1. Hence \(\mathcal S\) is surjective, completing the proof.
	\end{proof}
	
	\noindent{\bf Acknowledgments.}
	This work was supported by the National Natural Science Foundation of China 12271497. The author would like to thank Professor Lifeng Zhao for his guidance.
	
	\printbibliography
	
\end{document}